\pdfoutput=1
\documentclass[12pt,reqno]{amsart}

\numberwithin{equation}{section}
\usepackage[tt=false]{libertine}
\usepackage{mathtools}

\usepackage{amssymb,mathrsfs}
\usepackage[varbb]{newpxmath}

\let\savedbigtimes\bigtimes
\let\bigtimes\relax
\usepackage{mathabx} 
\let\bigtimes\savedbigtimes

\usepackage[margin=1in]{geometry}
\usepackage{graphicx}
\usepackage{enumerate}
\usepackage{bbm}
\usepackage{hyperref,color}
\usepackage[capitalize,nameinlink]{cleveref}
\usepackage[dvipsnames]{xcolor}
\hypersetup{
	colorlinks=true,
	pdfpagemode=UseNone,
    citecolor=OliveGreen,
    linkcolor=NavyBlue,
    urlcolor=black,
	pdfstartview=FitW
}
\usepackage{appendix}
\crefname{appsec}{Appendix}{Appendices}
\usepackage{tikz}
\usepackage{bm}
\usepackage{mathtools}

\newtheorem{theorem}{Theorem}[section]

\newtheorem{lemma}[theorem]{Lemma}
\newtheorem{corollary}[theorem]{Corollary}

\theoremstyle{definition}
\newtheorem{definition}[theorem]{Definition}
\newtheorem{example}[theorem]{Example}

\newtheorem*{assumption*}{Assumption}

\crefname{lemma}{Lemma}{Lemmas}
\crefname{theorem}{Theorem}{Theorems}
\crefname{definition}{Definition}{Definitions}
\crefname{fact}{Fact}{Facts}
\crefname{claim}{Claim}{Claims}
\crefname{proposition}{Proposition}{Propositions}

\newcommand{\E}{\mathbb{E}}
\newcommand{\Var}{\mathrm{Var}}

\renewcommand{\epsilon}{\varepsilon}

\newcommand{\Q}{\mathbb{Q}}

\renewcommand{\P}{\mathbb{P}}

\newcommand{\QQ}{\mathbb{Q}}

\newcommand{\beq}{\begin{equation}}
\newcommand{\eeq}{\end{equation}}

\newcommand{\cS}{\mathcal{S}}

\DeclareMathOperator{\Aut}{\textsf{Aut}}

\newcommand{\pcrit}{p_c}
\newcommand{\pE}{p_{\textsf{E}}}
\newcommand{\pmmt}{p_{\textup{\textsf{1M}}}}
\newcommand{\pAoN}{p_{\textup{\textsf{AoN}}}}
\newcommand{\Ind}[1]{\mathbf{1}\{#1\}}
\newcommand{\bZ}{\bm{Z}}
\newcommand{\mmse}{\textup{\textsf{MMSE}}}

\newcommand{\bH}{\bm{H}}
\newcommand{\bG}{\bm{G}}
\newcommand{\bJ}{\bm{J}}
\newcommand{\DKL}{D_{\textup{KL}}}
\newcommand{\Bern}{\textup{\textsf{Ber}}}

\newcommand{\BH}{}%{\color{blue!60!black}}
\newcommand{\BM}{}%{\color{magenta!90!red}}
\newcommand{\BG}{}%{\color{green!50!black}}
\newcommand{\BO}{}%{\color{orange}}
\newcommand{\BRR}{}

\newcommand{\EH}{}
\newcommand{\Btheta}{\hat{\theta}^\textup{B}}

\newcommand{\bS}{\bm{S}}
\newcommand{\bV}{\bm{V}}
\newcommand{\bY}{\bm{Y}}
\newcommand{\bP}{\mathbf{P}}

\newcommand{\PSI}{\psi}

\usepackage{appendix}
\crefname{appsec}{Appendix}{Appendices}

\begin{document}

\title{Sharp thresholds in inference of planted subgraphs}

\author[E.\ Mossel, J.\ Niles-Weed, Y.Sohn, N.\ Sun, I.\ Zadik]{Elchanan Mossel$^{\star\circ}$, Jonathan Niles-Weed$^\dagger$, Youngtak Sohn$^\star$, Nike Sun$^\star$, \\ and Ilias Zadik$^\star$}
\thanks{\raggedright$^\star$Department of Mathematics, MIT;
$^\circ$MIT Institute for Data, Systems, and Society;
$^\dagger$Center for Data Science \& Courant Institute of Mathematical Sciences, NYU. Email: \texttt{\{elmos,youngtak,nsun,izadik\}@mit.edu}; \texttt{jnw@cims.nyu.edu}}

\begin{abstract}%
We connect the study of phase transitions in high-dimensional statistical inference to the study of threshold phenomena in random graphs. 

A major question in the study of the Erd\H{o}s--R\'enyi random graph $G(n,p)$ is to understand the probability, as a function of $p$, that $G(n,p)$ contains a given subgraph $H=H_n$. This was studied for many specific examples of $H$, starting with classical work of Erd\H{o}s and R\'enyi (1960). More recent work studies this question for general $H$, both in building a general theory of sharp versus coarse transitions  (Friedgut and Bourgain 1999; Hatami, 2012) and in results on the location of the transition (Kahn and Kalai, 2007; Talagrand, 2010; Frankston, Kahn, Narayanan, Park, 2019; Park and Pham, 2022).

In inference problems, one often studies the  optimal accuracy of inference as a function of the amount of noise. In a variety of sparse recovery problems, an ``all-or-nothing (AoN) phenomenon'' has been observed: Informally, as the amount of noise is gradually increased, at some critical threshold the inference problem undergoes a sharp jump from near-perfect recovery to near-zero accuracy (Gamarnik and Zadik, 2017; Reeves, Xu, Zadik, 2021). We can regard AoN as the natural inference analogue of the sharp threshold phenomenon in random graphs. In contrast with the general theory developed for sharp thresholds of random graph properties, the AoN phenomenon has only been studied so far in specific inference settings, and a general theory behind its appearance remains elusive.

In this paper we study the general problem of inferring a graph $H=H_n$ planted in an Erd\H{o}s--R\'enyi random graph, thus naturally connecting the two lines of research mentioned above. We show that questions of AoN are closely connected to first moment thresholds, and to a generalization of the so-called Kahn--Kalai expectation threshold that scans over subgraphs of $H$ of edge density at least $q$. In a variety of settings we characterize AoN, by showing that AoN occurs \emph{if and only if} this ``generalized expectation threshold'' is roughly constant in $q$. Our proofs combine techniques from random graph theory and Bayesian inference.

\end{abstract}

% \begin{keywords}
% \BO planted subgraph model; all-or-nothing transition; Kahn--Kalai expectation threshold
% \end{keywords}

\maketitle

\date{\today}
\tableofcontents
\section{Introduction}

We consider the statistical model of a graph $H=H_n$ planted uniformly at random in an Erd\H{o}s--R\'enyi random graph $\bG\sim G(n,p)$. That is to say, the observation $\bY$ is the union of a uniformly random copy of $H$ in the complete graph $K_n$ (the signal) together with a sample from the Erd\H{o}s--R\'enyi measure $\Q_p=G(n,p)$ (the noise). Given the observation, the goal is to \emph{approximately recover} the hidden signal $H$, where recovery is measured in terms of the fraction of correctly recovered edges (see \eqref{e:mmse}). The model is formally specified in Definition~\ref{d:planted.subgr} below. This paper is concerned with the characterization of sharp information-theoretic thresholds in this inference problem.

% \subsection{Background}

Perhaps the most canonical such setting in the literature is the \emph{planted clique model}, where $H$ is a clique on $k$ vertices and $p=1/2$ \cite{Jer92}. It is a folklore result that \emph{exact recovery} of $H$ is possible when $k\ge (1+\epsilon)2\log_2 n$, and impossible when $k \le (1-\epsilon)2\log_2 n$. However, to the best of our knowledge, 
the natural question of whether one can recover a constant \emph{fraction} of $H$ when $k\leq (1-\epsilon)2 \log_2 n$ has not been previously considered, although strong impossibility results have been established in this regime for the slightly different detection framework \cite{Castro14}. We note that other specific choices of subgraphs have been studied in this literature, including the case where $H=H_n$ is a tree \cite{massoulie2019planting}, or $H=H_n$ a Hamiltonian cycle \cite{bagaria2020hidden}.

Obtaining a more refined understanding of statistical recovery guarantees in such models is further motivated by a growing body of recent work, initiated by \cite{gamarnik2022sparse} and \cite{reeves2021all}, which reveals that several high-dimensional Bayesian estimation models exhibit a sharp \textbf{``all-or-nothing'' (AoN) phase transition}: a very slight change in the signal-to-noise ratio separates a regime where one can recover almost all of the hidden signal (the ``all'' phase) from a regime where recovering even a constant fraction of the signal is impossible (the ``nothing'' phase). This is contrary to prior intuition derived from high-dimensional models where the transition is much smoother, e.g., compressed sensing or generalized linear models in the ``proportional regime'' \cite{reeves2019replica, barbier2019optimal}.

The underlying fundamental reasons why some inference models exhibit AoN, while others do not, remains --- to the best of our knowledge --- largely unknown. \BO This movitates us to study the general planted subgraph setting and ask: \EH
\begin{quotation}\centering{\emph{Which choices of hidden graphs $H=H_n$ lead to a sharp AoN transition?\\ \BO How do the graph theoretic properties of $H$ relate with sharp statistical phenomena? \EH  } }
\end{quotation}
Notably, the study of sharp thresholds for the occurence of specific subgraphs in the ``null'' (no hidden subgraph) Erd\H{o}s--R\'enyi random graph model $\Q_p=G(n,p)$ has a long and celebrated history dating back to \cite{MR125031}. This literature has recently led to the striking resolution of the Kahn--Kalai conjecture
\cite[Conjecture~1]{MR2312440}, which approximately locates the critical threshold for general monotone properties \cite{ MR4298747,park2022proof}. The community's notable understanding of transitions in the ``null'' model raises the possibility of better understanding information-theoretic transitions in the ``planted'' (hidden signal) models:
\begin{quotation}
\centering{\emph{How do the well-studied threshold phenomena in the ``null'' Erd\H{o}s--R\'enyi model\\ relate with statistical threshold phenomena in the ``planted'' model?}}
\end{quotation}
Our work is largely driven by this question.

\subsection{An overview}
In this paper, we aim to characterize the graph sequences $H=H_n$ for which the sharp \emph{all-or-nothing phenomenon} occurs in the associated inference problem. While the planted clique model is commonly formulated with fixed $p=1/2$ and $k$ varying, for general $H$ we adopt the more suitable perspective that we fix $H=H_n$ and vary $p$. Thus, given $H=H_n$, we ask whether there exists a critical value $\pAoN$ such that when $p \leq (1-\epsilon)\pAoN$ it is (information-theoretically) possible to recover an $(1-o(1))$-fraction of edges of the planted subgraph (the ``all'' phase), while when $p \geq (1+\epsilon)\pAoN$ it is impossible to recover any nontrivial fraction of the planted subgraph (the ``nothing'' phase). In other words, there is no intermediate ``something'' phase where one can recover a non-trivial fraction of the edges, but a non-trivial fraction must also be missed.

In this work we are able to characterize the occurrence of AoN in large families of planted subgraph models via a connection with a generalization of the \emph{expectation threshold} of \cite{MR2312440}. For a given $H=H_n$, the expectation threshold $\pE(H)$ is intended to approximate the critical threshold $\pcrit(H)$ at which $G(n,p)$ becomes likely to contain a copy of $H$. To make this more precise, for any graph $J$ we define the \emph{first moment threshold} $\pmmt(J)$ to be the minimum $p$ such that $G(n,p)$ contains at least one copy of $J$ \emph{in expectation}. The \emph{expectation threshold} $\pE(H)$ is the maximum first moment threshold among all subgraphs $J\subseteq H$, and the ``second Kahn--Kalai conjecture'' \cite[Conjecture 2]{MR2312440} posits that this is within a logarithmic factor of the threshold of interest $\pcrit(H)$.
The second Kahn--Kalai conjecture has been proved for bounded graphs $H$, see e.g. \cite[Theorem 4]{rucinski1987small}, but remains open in general.
\BO (It is not implied by the Kahn--Kalai conjecture for monotone properties that was mentioned above, see also \cite{mossel2022second}). \EH

% Intuitively, the second Kahn--Kalai conjecture corresponds to the natural idea that the existence of a copy of $H$ in $G(n,p)$ is driven by the existence threshold of its \BO ``least likely''  \EH subgraph, measured in terms of the largest first moment threshold. Indeed, for a graph to appear in $G(n,p)$, clearly all its subgraphs need to appear as well. Moreover, motivated by well-established results for bounded $H$ (e.g.\ \cite[Theorem 5]{rucinski1987small}) we expect also a ``clustering'' picture: whenever a copy of the ``least likely'' subgraph $J$ of $H$ appears in $G(n,p)$, we expect multiple copies of $H$ to appear as \emph{extensions} of the same subgraph $J$ as a core. Again, the second Kahn--Kalai conjecture (and the suggested ``sunflower'' picture) remains open, although several variants of it have already been proved~\cite{MR4298747,park2022proof,mossel2022second}.

In this work, we define for $q \in [0,1]$ the \emph{generalized expectation threshold} $\PSI_q(H)$ to be the maximum first moment threshold among all subgraphs $J\subseteq H$ such that $J$ contains at least a $q$ fraction of the edges of $H$ (Definition~\ref{d:gen.expectation.threshold}). Then $\PSI_0(H)$ is exactly the Kahn--Kalai expectation threshold. Our main result is a characterization of AoN for a large class of hidden graphs $H$ based on structural properties of $\PSI_q(H)$. We now state our main finding informally as follows:

\begin{theorem}[Main result, informally stated]
\label{thm:main_informal}
For various families of graph sequences $H_n$, the model of $H_n$ planted in $G(n,p)$ exhibits AoN if and only if $\PSI_q(H)$ is asymptotically constant as a function of $q \in (0,1)$.
\end{theorem}

\noindent See Theorems
\ref{t:dense.linear.aon} for the formal statement. As a corollary, we deduce for example that the model of a $k$-clique planted in $G(n,p)$ exhibits AoN if and only if $k$ is diverging with $n$ (see Corollary~\ref{c:planted.clique} in the Appendix). Our main result as stated above may not appear readily intuitive. For this reason, while stating our theorems in the following sections, we present several illustrative examples. Moreover, in Section~\ref{ss:intuition} we give general intuition by explaining how established results from random graph theory, alongside with the planting trick from the theory of random constraint satisfaction problems, are suggestive of such a connection.

% As a corollary, we conclude for example that AoN occurs whenever $H=H_n$ is a growing, balanced and dense graph (e.g., AoN occurs for the planted clique problem with diverging clique size).

It should be finally noted that a related, but incomparable, general investigation
was initiated by
\cite{huleihel2022inferring} for the information-theoretic limits of inferring a hidden
\emph{induced} subgraph in $G(n,p)$. By contrast, the observation $\bY$ in our setting is the \emph{union} of $G(n,p)$ with a hidden copy of $H$, i.e., the hidden copy need not be an induced subgraph of $\bY$. This difference makes our results incomparable; see also \cite[Section B.1]{huleihel2022inferring} for more discussion on differences between these two models.

% \iz{http://proceedings.mlr.press/v99/massoulie19a/massoulie19a.pdf} 

\subsection{Further motivations}
\label{ss:motivations}

As indicated above, the problem of recovering a hidden graph is directly connected to two major lines of research: 
\begin{enumerate}[1.]
\item \textbf{Threshold phenomena in random graphs.}
Starting from the work of \cite{MR125031}, a major research goal in the field of random graphs has been to understand, for any given graph $H$, for which values of $p$ is $H$ likely to appear in an Erd\H{o}s--R\'enyi graph $\bG\sim G(n,p)$. Note that $H$ may be a fixed graph, such as a triangle \cite{MR125031}, but it can also be a graph whose size and structure depends on $n$,
such as a perfect matching \cite{MR125031}  or Hamilton cycle \cite{MR389666,MR0434878}.
%Note that here there is no hidden copy of $H$: the question is to understand the probability that the ``noise'' graph $\bG\sim G(n,p)$ contains a copy of $H$. This probability is monotone in $p$, and we let $\pcrit(H)$ be the minimum value of $p$ for which the probability is at least $1/2$. 
Recent results on this question mainly go in one of two directions:

First, results in discrete Fourier analysis \cite{MR1645642,MR1678031,MR2925389} 
characterize general settings in which the transitions are sharp or coarse. A sharp transition means that the probability for $G(n,p)$ to contain a copy of $H$ is near zero for $p \le (1-\epsilon)\pcrit(H)$, and near one for $p \ge (1+\epsilon)\pcrit(H)$. A coarse transition means that the probability stays bounded away from zero and one for a non-trivial range of values $p \asymp \pcrit(H)$. The known characterizations of coarse thresholds can be interpreted as ``low complexity'' conditions, while sharp thresholds correspond to properties that do not have witnesses of low complexity. There has been a number of conjectures relating sharp thresholds in graphs to computational complexity, see e.g.~\cite{kalai2006perspectives}. 

A more recent line of research aims to roughly identify the location of the threshold in terms of (variants of) the Kahn--Kalai \textit{expectation threshold} \cite{MR2312440,MR2743011}. This  will be further discussed below. We note that the Fourier analysis literature mostly does not address the location of $\pcrit(H)$, while the expectation threshold literature mostly does not address the sharpness of the transition. The \BO optimal \EH results on the Kahn--Kalai conjectures are tight only up to logarithmic factors \cite{MR4298747,park2022proof}.

\item \textbf{Threshold behavior of inference problems (all-or-nothing phenomena).} AoN was first identified in the context of sparse linear regression~\cite{gamarnik2022sparse, reeves2021all}, and has since been established for numerous other models, including sparse tensor PCA~\cite{AoNtensor}, Bernoulli group testing~\cite{truong2020all,niles2021all,coja2022statistical}, and random graph matching~\cite{wu2022settling}. A striking observation is that AoN arises for several models which are conjectured to exhibit a statistical-computational gap, although no rigorous connections currently exist.

The setting studied by~\cite{macris2020all,AoNtensor} consists of Bayesian inference problems where one observes a rank-one spike corrupted by gaussian noise. 
%In this case, AoN manifests as a sharp discontinuity of the limiting minimum mean squared error (MMSE), viewed as a function of the noise variance $\sigma^2$. 
In this setting, \cite{AoNtensor} give general \emph{sufficient} conditions on the prior distribution under which AoN occurs, \BO as a function of the noise variance $\sigma^2$. \EH These conditions amount to a quantitative ``anti-concentration'' requirement that independent draws from the prior are unlikely to be highly correlated. 

The models considered by~\cite{reeves2021all,truong2020all,luneau2022information,niles2021all,coja2022statistical} are generalized linear models. As in the gaussian setting, when the prior satisfies a suitable anti-concentration condition, AoN occurs \BO (here, as a function of the number of observations). \EH
AoN for a Bernoulli model with added noise was established by~\cite{wu2022settling}, who consider the problem of recovering the correspondence between a pair of correlated random graphs with randomly permuted vertex labels. In this last case, AoN arises as a function of the random graph density.
\end{enumerate}
The similarities between sharp thresholds and AoN phenomena are quite clear. However, the settings are very different. In the random graph setting we are looking for a graph that appears at random, while in inference, there is a true signal that is corrupted by random noise. In this paper we connect the two by studying the inference problem in the random graph setting. A major theme of this paper is revealing that the  all-or-nothing phenomenon in inference models are actually closely connected with the behavior of (variants of) the expectation threshold in the corresponding null models; see below for more discussion.

From a \textbf{technical standpoint}, the problem studied in this paper also naturally brings together the two communities: one studying properties of random graphs and the other studying inference in high-dimensions. Interestingly, the proofs in the paper combine ideas and techniques from both communities: some of our conditions are stated in terms of the expectation thresholds of subgraphs of varying sizes, thus refining definitions from \cite{MR2312440}. As customary in the study of expectation of thresholds, our proofs use variants of the second moment method but also recent ideas in the community, such as the one used for the ``spread lemma'' \cite{sunflower_annals, MR4298747}. However, we also heavily use the Bayesian perspective, in particular the planting trick that we borrow from the study of planted constraint satisfaction problems, see also \cite{achlioptas2008algorithmic,coja2022statistical}. Other concepts from high dimensional inference such as the I-MMSE relation from information theory \cite{guo2011estimation} and Nishimori identity from statistical physics \cite{nishimori2001statistical} also play a role in some of the proofs. 

We finally note a related line of work in coding theory. The goal of coding theory is to recover code words sent over a noisy channel. The analogy to the model we study in the paper is quite clear. Our ``codewords'' are the planted graphs, and the ``channel'' is the operation of taking a union with a sample from the Erd\H{o}s--R\'enyi measure. Many results in coding theory can be viewed as AoN statements (but for probability error metrics, not partial recovery metrics) as they show that the \emph{decoding error probability} jumps sharply from zero to one as a function of the channel noise. A recent striking example was established in~\cite{kudekar2016reed} for the Reed--Muller code and the erasure channel. An interesting aspect of their example is the use of variants of the KKL theorem \cite{kahn1988influence} from discrete Fourier analysis to prove a sharp threshold corresponding to AoN which then in turn allow to prove that these codes achieve the capacity of the channel.

\section{Main results}
\label{ss:main.results}

In this section we state our first main result, Theorem~\ref{t:dense.linear.aon}, which characterizes the occurrence of AoN (at linear scale) in sufficiently dense graphs under mild technical assumptions, and describe in high level the rest of our results. This section is organized as follows:
\begin{itemize}
\item In \S\ref{ss:gen.exp.thresholds} we formalize the planted subgraph model, and define the generalized expectation thresholds which were informally introduced above.
\item In \S\ref{ss:statement.t.aon.dense} we give the statement of Theorem~\ref{t:dense.linear.aon}, along with several motivating examples.
\item In \S\ref{ss:beyond} we give an overview of some of our results beyond the setting of AoN at linear scale for sufficiently dense graphs. We present also several relevant examples.
\item In \S\ref{ss:intuition} we offer some intuition behind our main result, based on known results in the random graphs literature.
\end{itemize}

\subsection{Generalized expectation thresholds}
\label{ss:gen.exp.thresholds}

We begin by formalizing the planted subgraph model discussed above:

\begin{definition}[planted subgraph model]
\label{d:planted.subgr}
Let $H=H_n$ be a given graph. We will abbreviate $v(H) \equiv v_H$ for the number of vertices of $H$, and $|H|\equiv e(H) \equiv e_H$ for the number of edges of $H$. We always assume $v(H) \le n$.
Let $\cS$ be the set of (isomorphic) copies of $H$ in the complete graph $K_n$,
and let $\bP$ be the uniform probability measure over $\cS$. We work on the so-called \emph{planted model} $\P_p$, the observation is $\bY=\bH\cup\bG$ where $\bH\sim\bP$ and $\bG$ is an independent sample from $G(n,p)$. The goal is to recover $\bH$ from $\bY$. For comparison, we also introduce the \emph{null model} $\Q_p$ where there is no hidden $\bH$, and we observe simply $\bY=\bG\sim G(n,p)$.
\end{definition}

Throughout the paper we identify all graphs on $n$ vertices, e.g., the instances of $\bY, \bH, \bG$, with their naturally corresponding binary vectors in $\{0,1\}^{\binom{n}{2}}.$ Our measure of ``recovery'' in the planted subgraph model is the fraction of correctly recovered edges, which naturally corresponds to the \emph{minimimum mean squared error} (MMSE):
	\beq\label{e:mmse}
	\mmse_N(p)
	\equiv\E_{\P_p}
	\bigg[\Big\|\bH-\E[\bH|\bY]\Big\|^2_2\bigg]
	= 
	\E_{\P_p} \bigg[\|\bH\|^2 - \Big\|\E[\bH|\bY]\Big\|^2\bigg]
	\,.\eeq
Since $\|\bH\|^2=e(H)$ almost surely, the MMSE must lie in $[0,e(H)]$, so we always normalize it by $e(H)$ in what follows. It is well known that the MMSE is nondecreasing in $p$, and we review the short proof in Lemma~\ref{l:mmse.nondec} below. In light of this, it is natural to ask in what situations the MMSE has a \emph{sharp} transition. To this end, following the AoN literature, we make the following definition:

\begin{definition}[all-or-nothing]\label{d:aon.linear.subgraph}
We say that the model from Definition~\ref{d:planted.subgr}
exhibits an \emph{all-or-nothing (AoN) transition} at critical probability $\pAoN=\pAoN(N)$ if
\[
\lim_{N \to \infty}  \frac{\mmse_N(p)}{e(H)}
= \left\{
\begin{array}{ll}
1 & \textup{for all $p\ge(1+\epsilon)\pAoN$} \\
0 & \textup{for all $p\le(1-\epsilon)\pAoN$}\,,
\end{array}
\right.
\]
for any constant $\epsilon>0$. 
\end{definition}

 As mentioned above, a theme of this paper is that the location of $\pAoN$ is closely connected to first moment thresholds of the subgraphs $J$ of $H$ in the null model $\QQ_p$. For a given subgraph $J$ of $K_n,$ the first moment threshold $\pmmt(J)$ is the smallest value of $p$ such that the expected copies of $J$ in a sample from the null model $\Q_p$ is at least one. More formally, for $\bG \sim G(n,p)$ let $\bZ(\bG)$ denote the number of copies of $J$ that are contained in $\bG$.
We define the \emph{first moment threshold}
$\pmmt(J)$ to be the value of $p$ that satisfies
	\beq\label{def:p1M_colt}
	1 = \E_{\Q_p} \bZ(\bG)
	= M_J p^{e(J)}\,,
	\end{equation}
that is, $\pmmt(J)\equiv M_J^{-1/e(H)}$, where $M_J\equiv M_{J,K_n}$, the number of copies of $J$ in $K_n$.

In Theorem~\ref{t:dense.linear.aon} below, we characterize the occurrence of AoN for a class of planted subgraphs $H$ that are sufficiently dense, meaning more precisely that $v(H)\to\infty$ and
    \beq\label{e:intro.dense}
  e(H) \gg v(H)\log v(H)
    \eeq
in the limit $n\to\infty$ (Definition~\ref{d:dense}). 

We start our quest with identifying the AoN threshold $\pAoN$ for this class of planted subgraphs. In view of the classical random graph theory, the first natural attempt would be to ask whether $\pAoN$ coincides with the first moment threshold $\pmmt(H)$ \eqref{def:p1M_colt} of the subgraph $H$. The answer is no, as we illustrate with Example~\ref{example:sun} below. In fact, the possibility of $\pAoN \ne \pmmt(H)$ is closely related to ideas underlying the Kahn--Kalai expectation threshold \cite{MR2312440} and the threshold $\pAoN$ found by Theorem~\ref{t:dense.linear.aon} turns out to indeed be a variant of the Kahn--Kalai expectation threshold. 

% \textbf{The main result of this section is Theorem~\ref{t:dense.linear.aon},} which characterizes AoN for a class of graphs $H$ that are \BO sufficiently dense, \EH meaning more precisely that $v(H)\to\infty$ and
%     \beq\label{e:intro.dense}
%     \frac{e(H)}
%     {v(H)\log v(H)} \gg1
%     \eeq
% in the limit $n\to\infty$ (Definition~\ref{d:dense}). Moreover, the threshold $\pAoN$ found by Theorem~\ref{t:dense.linear.aon} turns out to be a variant of the Kahn--Kalai expectation threshold. 

Let us take a moment to discuss the first moment threshold \eqref{def:p1M_colt} in more detail.
Note that $M_H\equiv M_{H,K_n}$ is the number of copies of $H$ in $K_n$:
	\beq\label{e:intro.M.H}
	M_H 
	= \frac{n! / (n-v_H)!}{|\Aut H|}
	= \frac{(n)_{v_H}}{|\Aut H|} \le n^{v(H)}\,,
	\eeq
where $\Aut(H)$ denotes the automorphism group of $H$. It follows that
    \beq\label{e:graph.first.mmt.threshold}
	\pmmt(H) \equiv  \bigg(\frac{1}{M_H}\bigg)^{1/e(H)}
	\ge \frac{1}{n^{v(H)/e(H)}}
	\eeq
for any $H$; and the sharpness of this trivial lower bound depends on the size of the automorphism group of $H$. However, for graphs that are \BH sufficiently dense (in the sense of \eqref{e:intro.dense}), \EH  the factor $|\Aut H| = v(H)^{O(v(H))}$ becomes negligible when raised to the power $1/e(H)$, so for such graphs
we obtain the simplification
	\beq\label{e:dense.threshold}
	\pmmt(H) = \frac{1 +o_n(1)}{n^{v(H)/e(H)}}\,.
	\eeq
This is formalized in Lemma~\ref{l:dense.threshold} below.

\begin{example}[AoN and first moment thresholds can differ]
\label{example:sun}
Let $J$ be a clique on vertices $\{1,\ldots,k\}$. We let $H$ be obtained from $J$ as follows: take $k$ additional vertices $\{k+1,\ldots,2k\}$, and form an edge between vertex $i$ and vertex $k+i$ for each $1\le i\le k$. Thus $H\supseteq J$, $v(H)=2k$, and $e(H)=e(J)+k$. We consider the model of $H$ planted in $G(n,p)$ (Definition~\ref{d:planted.subgr}).
Assume $k=k_n\to\infty$, and note that $J$ contains most of the edges of $H$. Since we are interested in edge recovery we expect that the inference problem for $H$ exhibits AoN at $\pAoN(J)=\pmmt(J)$ (cf.\ Corollary~\ref{c:planted.clique}); and indeed we prove this in \BH Theorem~\ref{t:dense.linear.aon} below (see also Example~\ref{example:sun:generalize}). \EH The first moment threshold for $H$ is much smaller than that of $J$, so it does not match the AoN transition:
	\[
	\pmmt(H)
	= \frac{1+o_n(1)}{n^{4/(k+1)}}
	\ll 
	\frac{1+o_n(1)}{n^{2/(k-1)}}
	=\pmmt(J)
	=\pAoN(J)\,.
	\]
(The threshold $\pmmt(H)$ can be obtained by direct calculation, or by appealing to \eqref{e:dense.threshold} or Lemma~\ref{l:dense.threshold}.)
\end{example}

The problem illustrated by Example~\ref{example:sun} is closely related to the ideas underlying the Kahn--Kalai conjectures.
Recall from above that the basic assertion of these conjectures is that
while $\pcrit(H)$ may be far from $\pmmt(H)$, there must be a subgraph $J\subseteq H$ for which $\pmmt(J)$ is not too far from $\pcrit(H)$. 
Clearly, this is highly analogous to Example~\ref{example:sun}, where the AoN transition is driven by the clique $J\subseteq H$.
That is to say, the transition can be estimated by
the \emph{expectation threshold}
	\beq\label{e:pE}
	\pE(H)
	= \max\Big\{
	 \pmmt(J)
	 : \varnothing \subsetneq
	 J \subseteq H\Big\}\,.
	\eeq
In particular, the ``second Kahn--Kalai conjecture'' \cite[Conjecture~2]{MR2312440} posits that
	\beq\label{e:second.kk}
	\pE(H) \lesssim
	\pcrit(H)
	\lesssim \pE(H) \log |H|\,,
	\eeq
where the lower bound is trivial, and the logarithmic factor is known to be necessary. We discuss this further in Example~\ref{x:matching} below.

Analogously to the Kahn--Kalai conjectures, it is natural to ask whether $\pAoN(H)$ is related to $\pE(H)$. In this paper we study the question of locating $\pAoN(H)$ up to $1+o_n(1)$ factors, as opposed to logarithmic factors. At this level of precision, it turns out that $\pAoN(H)$ does not necessarily coincide with $\pE(H)$, as illustrated by Example~\ref{example:aon:diff:expect} below. One reason is that, in the context of AoN, since we are interested in recovery of \emph{almost} all or \emph{almost} none of the edges, we expect that only linear-sized subgraphs of $H$ should be relevant to the transition. For this reason,
we define a slight generalization of the expectation threshold which turns out to be more relevant to the AoN question: 

\begin{definition}[generalized expectation threshold]
\label{d:gen.expectation.threshold}
For $q\in [0,1]$ and a given graph $H=H_n$, define the \emph{$q$-constrained expectation threshold} to be the largest first moment threshold among subgraphs of $H$ with at least $q$ fraction of the edges. That is,
	\beq\label{e:psi.q}
	\PSI_q\equiv
	\PSI_q(H)
	\equiv \max \bigg\{ \pmmt(J) :
	J\subseteq H\textup{ with }
	|J| \ge\max\Big\{1, |H| q\Big\}
	\bigg\}\,.\eeq 
In particular, $\PSI_0(H)$ is the same as the Kahn--Kalai expectation threshold $\pE(H)$.
\end{definition}

\subsection{Statement of AoN characterization for sufficiently dense graphs} 
\label{ss:statement.t.aon.dense}

The following theorem \emph{characterizes} AoN for sufficiently dense graphs $H$, subject to the additional technical requirement that $H$ must be ``delocalized,'' meaning roughly that $H$ does not contain a sublinear sized subgraph that is particularly dense. More precisely, we require that $H$ must contain a subgraph $J$ with $e(J)/e(H)\ge q$, such that $J$ has nearly maximal density among all subgraphs of $H$:
	\beq\label{e:intro.delocalized}
	\frac{v(J)}{e(J)}
	\le
	\min\bigg\{
	\frac{v(J')}{e(J')} :\varnothing\subsetneq J'\subseteq H
	\bigg\}
	+ \frac{C}{\log n}\,,
	\eeq
where $q$ and $C$ are constants not depending on $n$. If $H=H_n$ is dense and satisfies $\PSI_q(H) \ge c \cdot \PSI_0(H)$ for positive constants $q$ and $c$, then $H$ is delocalized; see Definition~\ref{d:delocalized} for details. For dense delocalized graphs, we have the following result, which greatly generalizes Example~\ref{example:sun}: 

\begin{theorem}[characterization of AoN for sufficiently dense graphs]\label{t:dense.linear.aon}
Suppose $H=H_n$ is sufficiently dense (\eqref{e:intro.dense} or Definition~\ref{d:dense}) and delocalized (\eqref{e:intro.delocalized} or Definition~\ref{d:delocalized}).
Then the model of $H_n$ planted in $G(n,p)$ exhibits AoN if and only if 
	\beq\label{e:almost.bal}
 	\lim_{n\to\infty} \frac{\PSI_q(H_n)}{\PSI_{q^\prime}(H_n)}=1
	\quad\textup{ for all }q,q'\in(0,1)\,.
 	\eeq
Moreover, in this case $\pAoN = (1+o(1))\PSI_q(H)$ for any $q\in(0,1)$. 
\end{theorem}

We say that $H_n$ is \emph{almost balanced} if it satisfies condition~\eqref{e:almost.bal} (this derives from the terminology of balanced graphs; see Example~\ref{x:balanced} below). We illustrate Theorem~\ref{t:dense.linear.aon} with the following:

\begin{example}[generalization of Example~\ref{example:sun}]
\label{example:sun:generalize}
Suppose $H=H_n$ is sufficiently dense (as in \eqref{e:intro.dense}), and that there is a subgraph $J\subseteq H$ such that (i) $J$ contains most of the edges of $H$, 
	$|J|/|H| = 1-o_n(1)$,
and (ii) the first moment threshold of $J$ captures the expectation threshold of $H$,
    \[\frac{\pmmt(J)}{\pE(H)}
	=\frac{\pmmt(J)}{\PSI_0(H)}
	= 1-o_n(1)\,.\]
Then, for any $q<1$, it follows 
from Definition~\ref{d:gen.expectation.threshold}
that $
	\PSI_0(H)
	\ge \PSI_q(H)
	\ge \pmmt(J)
	= (1-o_n(1)) \PSI_0(H)\,.$
This implies that $H$ is delocalized (see Definition~\ref{d:delocalized}), and satisfies condition~\eqref{e:almost.bal}.
Therefore, Theorem~\ref{t:dense.linear.aon} 
implies that the model of $H_n$ planted in $G(n,p)$ exhibits AoN at $\pAoN(H)=\pmmt(J)$. This generalizes Example~\ref{example:sun}.
\end{example}
%shows that $H_n$ exhibits AoN at $\pAoN=(1+o(1))\pmmt(H_n^\star)$. The result stated in Example \ref{example:sun} can be seen by taking $H_n^\star$ to be the $k$-clique inside $H$.

\begin{example}[dense balanced graphs]
\label{x:balanced}
Suppose $H=H_n$ is sufficiently dense (in the sense of \eqref{e:intro.dense}), and \emph{balanced} in the sense that $H$ has maximal edge density among all its subgraphs \cite{MR125031,MR620729}: 
	\[\frac{e(H)}{v(H)}
	= \max \bigg\{ \frac{e(J)}{v(J)} : J \subseteq H\bigg\}\,.
	\]
This implies that $H$ is delocalized (\eqref{e:intro.delocalized} or Definition~\ref{d:delocalized}). It follows from \eqref{e:dense.threshold} or
Lemma~\ref{l:dense.threshold} that we have $\PSI_q(H)=(1+o_n(1)) \pmmt(H)$ for all $q\in(0,1]$, 
so the almost-balanced condition \eqref{e:almost.bal} is satisfied. Therefore, it follows by Theorem~\ref{t:dense.linear.aon} that 
the model of $H$ planted in $G(n,p)$
 exhibits AoN at $\pAoN(H)=\pmmt(H)$. This implies the AoN result for the planted clique model (also established by another argument in the  Corollary \ref{c:planted.clique} in the Appendix).
\end{example}

Given Examples~\ref{example:sun:generalize} and \ref{x:balanced}, one might ask if it is true that $\pAoN$ always equals the expectation threshold.
%\BM In Example~\ref{x:balanced}, we leave open the possibility that $\pE(H)$ can be much larger than $\pmmt(H)$. \EH
We next present a simple example where the AoN and expectation thresholds differ:

\begin{example}[AoN and expectation thresholds can differ]\label{example:aon:diff:expect}
Let $H$ be the disjoint union of $J_0,\ldots,J_k$ where $J_0$ is a $(2k)$-clique while $J_i$ is a $k$-clique for each $1\le i\le k$, with $k= \log n$. Then $H$ is dense, since $v(H) = 2k + k^2 = (1+o_n(1)) k^2$, while
    \[
    e(H) = \binom{2k}{2} + k \binom{k}{2}
    = (1+o_n(1)) \frac{k^3}{2}
    \gg v(H) \log v(H)\,.
    \]
Note that $J_0$ is the densest subgraph of $H$,
with $e(J_0)/v(J_0) = (1+o_n(1)) k$. Using \eqref{e:dense.threshold} or Lemma~\ref{l:dense.threshold}, we have the lower bound
    \[
    \pE(H)=\PSI_0(H)
    \ge \pmmt(J_0)
    \stackrel{\eqref{e:dense.threshold}}{=}
    \frac{1+ o_n(1)}{n^{v(J_0)/|J_0|}}
    =\frac{1+ o_n(1)}{n^{2/(2k-1)}}
    =\frac{1+ o_n(1)}{e}
    \,.
    \]
On the other hand, $J_0$ accounts for only a negligible fraction of the edges of $H$. We will argue that if $|J|/|H|$ is lower bounded by any positive constant, then the density of $J$ cannot be much larger than $k/2$. To this end, let us decompose $v(J) = v_0 + \ldots v_k$ where $v_i$ is the number of vertices in $J\cap J_i$. Then
    \[
    |J|
    \le \sum_{i=0}^k \frac{(v_i)^2}{2}
    \le 2k \cdot \frac{v_0}{2}
    +k \sum_{i=1}^k \frac{v_i}{2}
    = \frac{k (v_0 + v(J))}{2}
    \le \frac{k (2k + v(J))}{2}
    \,.
    \]
We also trivially have $|J| \le v(J)^2/2$, so in order for $|J| \ge |H| q$ we must have $v(J) \gg k$.
It follows that for all $|J| \ge |H|q$, the right-hand side above is roughly $kv(J)/2$, and therefore
    \[
    \frac{|J|}{v(J)}
    \le \frac{k (2k + v(J))}{2 v(J)}
    = \frac{[1+ o_n(1)]k}{2}\,.
    \]
Moreover, the bound is clearly asymptotically achieved by taking most of the $v_i$ to be either zero or $k$. It follows that for $q>0$ we have
    \[
    \PSI_q(H)
    = \frac{1+o_n(1)}{n^{[1+o_n(1)]2/k}}
    = \frac{1+o_n(1)}{e^2}\,.
    \]
It is straightforward to check that $H$ is delocalized (\eqref{e:intro.delocalized} or Definition~\ref{d:delocalized}), so it follows from Theorem~\ref{t:dense.linear.aon} that this model exhibits AoN at  $\pAoN(H)=1/e^2$, which is smaller than the expectation threshold $\pE(H)$. 
\end{example}
% Finally, we note that the above bound on $\pE(H)$ is asymptotically tight. \EH Since $J_0$ is the densest subgraph of $H$, it is clear from \eqref{e:dense.threshold} that $\pmmt(J) \le [1+o_n(1)]/e$ for any dense subgraph $J\subseteq H$. We can thus restrict our attention to the case $|J| = O(v_J \log v_J)$. In this we can use \eqref{e:intro.M.H} to bound
%     \[
%     \pmmt(J)
%     \le 
%     [1+o_n(1)]
%     \frac{v(J)^{v(J)/|J|}}{n^{v(J)/|J|}}
%     = [1+o_n(1)]
%     \exp\bigg\{
%     \frac{v_J\log v_J}{|J|}
%     \bigg(\frac{\log n}{\log v_J}-1\bigg)
%     \bigg\}\ll1\,,
%     \]
% since $\log v_J \ll \log n$. This verifies that $\pE(H)=\PSI_0(H) = [1+o_n(1)]/e^2$, as claimed. \EH

\subsection{Results beyond dense graphs}
\label{ss:beyond}

One reason that (sufficiently) dense graphs are easier to analyze is that in this case, the first moment threshold \eqref{def:p1M_colt} can be approximated by the much simpler expression \eqref{e:dense.threshold}. We do not have a similarly strong characterization for general graphs, where we expect the order of the automorphism group to play a role. Indeed, in Example~\ref{x:cycle.out} below, we show that the dense assumption is necessary in Theorem~\ref{t:dense.linear.aon}.
However, in this work we also present results that are able beyond the dense graphs regime under the following assumptions:
\begin{enumerate}[1.]
\item \textbf{A general ``nothing'' phase} In \S\ref{ss:spread} we give a general ``nothing'' result that applies to all planted subgraphs $H$. In Theorem~\ref{t:nothing.spread.subgraph} we prove that when $p\gg \PSI_q(H)$ for all $q>0$, where $\PSI_q$ is the generalized expectation threshold of Definition~\ref{d:gen.expectation.threshold} ``nothing'' holds for general graphs $H$ (i.e., $\lim_{N \to \infty} \mmse_N(p)/e(H)=1$). This can be interpreted as an approximate variant of the second Kahn--Kalai conjecture, since for ``nothing'' to appear the ``noise'' $G(n,p)$ much have created an approximate copy of $H$ \BO that is nearly disjoint from the signal. \EH See Example~\ref{x:matching} below for further discussion and \S\ref{ss:spread} for more relevant references. 

\item \textbf{AoN for small sparse graphs} In \S\ref{ss:sparse} we prove a result for the sparse regime. In Theorem~\ref{t:aon.sparse} we proves AoN for models where the planted subgraph is small (of size $O(\log n/\log \log n)$), sparse, and \emph{strongly balanced} (Definition~\ref{d:strongly.balanced}).
The latter is a slight generalization of the notion introduced by
\cite{MR890231}, and is more restrictive than the balanced condition appearing in Example~\ref{x:balanced}. It follows from Theorem~\ref{t:aon.sparse} that if $H$ is a small tree or cycle --- more precisely, if $v(H)$ satisfies the bound \eqref{e:tree.cycle.bound} --- 
then the model of $H$ planted in $G(n,p)$ exhibits AoN at $\pAoN(H)=\pmmt(H)$ (see Example~\ref{x:small.tree.cycle}).

\item \textbf{AoN in the exponential scale} In \S\ref{ss:prelim.exp} and Section~\ref{s:exp} we consider AoN phenomena at exponential scale (Definition~\ref{d:aon.exp}) rather than linear scale (Definition~\ref{d:aon.linear.subgraph}), meaning the transition is in terms of $\log p$ rather than $p$ itself (when $p>\pAoN(H)^{1-\epsilon}$ ``nothing'' holds, but when $p<\pAoN(H)^{1+\epsilon}$ ``all'' holds). This relaxation of the AoN phenomenon allows us to establish a characterization without any density assumption, and therefore taking into account the automorphism group of $H$. In Theorem~\ref{t:aon.exp}, we show that under a mild technical condition of being ``first-moment-stable'' (Definition~\ref{d:first.mmt.stable}), AoN occurs at the exponential scale if and only if $H=H_n$ is ``first-moment-flat,'' meaning that
    \beq\label{e:intro.first.mmt.flat}
    \frac{\log\PSI_q(H_n)}{\log\PSI_{q'}(H_n)}
    = 1+o_n(1)
    \eeq
for all $q,q'\in(0,1)$ (see Definition~\ref{d:first.mmt.flat}, and compare with the almost-balanced condition \eqref{e:almost.bal}). We illustrate this with Example~\ref{x:small.bal} below. We also prove a location result, Theorem~\ref{thm:locating}, which says that the AoN and first moment thresholds must always coincide at the exponential scale.
\end{enumerate}

\begin{example}[perfect matchings]
\label{x:matching}
Let $H$ be a perfect matching on $n$ vertices, so the number of edges is $|H|=n/2$. 
Recovering a matching planted in an Erd\H{o}s--R\'enyi graph was proposed in the physics literature as a toy model for particle tracking~\cite{chertkov2010inference,MR4147946}. Subsequent work has rigorously analyzed planted matching recovery in closely related models, with remarkably precise results \cite{MR4350971,ding2021planted}. In particular, AoN generally does not occur in planted matching models. We note here that our general theorems, which are not tailored to the matching problem, nevertheless indicate weaker results of a similar flavor. To see this, let $J$ be any subgraph of $H$ with $|J|=|H|q$. Then
    \[M_J
    \stackrel{\eqref{e:intro.M.H}}{=}
    \frac{(n)_{v(J)}}{|\Aut J|}
    = \frac{n!/(n-nq)!}{2^{nq/2} (nq/2)!}\,.
    \]
It follows by Stirling's formula that
    \[
    \pmmt(J)
    = \frac{1}{(M_J)^{2/(nq)}}
    = \frac{[1+ o_n(1)]e}{n}
    \exp\bigg\{
    \log q + \frac{2(1-q)}{q} \log(1-q)
    \bigg\}
    \,.
    \]
The exponent is an increasing function on $0\le q\le 1$, so we conclude
$\PSI_q(H)= [1+o_n(1)]e/n$
for all $q\in[0,1]$. One can then check that the conditions of Theorem~\ref{t:aon.exp} are satisfied, so we can conclude AoN at the \emph{exponential} scale at $\pmmt(H)=1/n$.
Theorem~\ref{t:nothing.spread.subgraph} allows us to say something slightly more in one direction, namely that $p \gg 1/n$ is in the ``nothing'' regime for this model. However, for all $p\asymp 1/n$, the total number of edges in the observed graph will be of order $n$, so even a random subset of $n/2$ of the observed edges will have non-trivial overlap with the hidden matching. For this reason we expect that the model has an ``all'' phase for $p\ll 1/n$, a ``something'' phase around $p\asymp 1/n$, and a ``nothing'' phase for $p \gg 1/n$.

Let us also note that, in the context of the second Kahn--Kalai conjecture, it is known that $\pE(H)=\PSI_0(H)\asymp 1/n$ (as confirmed by the above calculation), but $\pcrit(H) \asymp (\log n)/n$. Indeed, in order to contain a perfect matching the graph must have minimum degree at least one, and the coupon collector effect is responsible for the $\log n$ factor. Moreover, this is the reason for the logarithmic factor in the second Kahn--Kalai conjecture \eqref{e:second.kk} (see the discussion of \cite[\S2]{MR2312440}).%
\end{example}
% [example of Theorem~\ref{t:nothing.spread.subgraph}]

\begin{example}[small balanced graphs]\label{x:small.bal} 
We saw in Example~\ref{x:balanced} that dense balanced graphs exhibit AoN at the first moment threshold. If $H$ is balanced and sufficiently small, $v(H) \le n^{o(1)}$, but not necessarily dense, we can obtain the weaker result that AoN occurs at the exponential scale. Indeed, for $v(H) \le n^{o(1)}$, similar considerations as \eqref{e:dense.threshold} give
    \[
    \log\frac{1}{\pmmt(H)}
    = \frac{v(H)}{e(H)}\log n
    -O\bigg( \frac{v(J) \log v(J)}{e(H)}\bigg)
    = [1+o_n(1)]\frac{v(H)}{e(H)}\log n\,.
    \]
If the graph is balanced, then it follows that
    \[
    \log\PSI_q(H)
    = [1+o_n(1)]\frac{v(H)}{e(H)}\log n
    \]
for all $q\in[0,1]$. This implies that the conditions of Theorem~\ref{t:aon.exp} are satisfied, so we have AoN at the exponential scale at $p=n^{-v(H)/e(H)}$. \EH
\end{example}

\begin{example}[cycle with out-edges]
\label{x:cycle.out}
This example shows the necessity of the ``dense'' assumption in 
Theorem~\ref{t:dense.linear.aon}.
Let $H$ be formed by a cycle $J$ on $k$ vertices, together with one extra outgoing edge for each vertex of the cycle, so that in total $|H|=v(H)=2k$. Assume $k=\log\log n$, so that the model of $J$ planted in $G(n,p)$ has AoN at $\pAoN(J)=1/n$  by Example~\ref{x:small.tree.cycle}. 
It is not too difficult to verify that $\PSI_q(H)=(1+o(1))n^{-1}$ for every fixed $q\in [0,1]$ (see Example~\ref{x:cycle} for the details of calculating $\PSI_q$ for the $k$-cycle, which is similar). Therefore $H$ is delocalized (see Definition~\ref{d:delocalized}) and almost-balanced (condition~\eqref{e:almost.bal}). However, we claim that the model of $H$ planted in $G(n,p)$ does not exhibit AoN at the linear scale (Definition~\ref{d:aon.linear.subgraph}): for $p \le 0.9/n$, we are below $\pAoN(J)$, so we expect to be able to recover most of $J$. However, for $p \ge 0.5/n$, a linear fraction of the vertices of $J$ will typically have more than one outgoing edge in $G(n,p)$, meaning we will fail to recover a constant fraction of the edges in $H\setminus J$. It follows that for $0.5/n \le p \le 0.9/n$ we have neither ``all'' nor ``nothing,'' so AoN at linear scale (Definition~\ref{d:aon.linear.subgraph}) does not occur.  However, the conditions of Theorem~\ref{t:aon.exp} are satisfied (indeed, this is a special case of Example~\ref{x:small.bal}), so we do have AoN at the exponential scale.
\end{example}

\subsection{Intuition from random graph theory and proof outline}
\label{ss:intuition}

We close by offering some intuition behind our main result connecting AoN with the generalized expectation thresholds $\PSI_q(H)$ (see the informal statement of Theorem~\ref{thm:main_informal}, or the formal statements of Theorems 
\ref{t:dense.linear.aon} and \ref{t:aon.exp}).

We start with the intuition from the theory of random graphs. Recall the 
``second Kahn--Kalai conjecture'' \eqref{e:second.kk}, which
estimates $\pcrit(H)$ in terms of the expectation threshold $\pE(H)$.
This corresponds to the natural idea that the existence of a copy of $H$ in $G(n,p)$ is driven by the existence threshold of its ``least likely'' subgraph, measured in terms of the largest first moment threshold. Indeed, for a graph to appear in $G(n,p)$, clearly all its subgraphs need to appear as well. Moreover, motivated by well-established results for bounded $H$ (e.g.\ \cite[Theorem 5]{rucinski1987small}) we expect also a ``clustering'' picture to emerge in $G(n,p)$: whenever a copy of the ``least likely'' subgraph $J$ of $H$ appears in $G(n,p)$, we expect multiple copies of $H$ to appear as \emph{distinct extensions} of the same subgraph $J$ as a core (leading to a ``sunflower'' structure). For the interested reader, we remark that the sunflower picture corresponds to a ``condensation phase'' in the language of random constraint satisfaction problems \cite{krzakala2007gibbs}. We note however that the second Kahn--Kalai conjecture (and the suggested ``sunflower'' structure) remains open for general $H$, although several variants of it have been proved~\cite{MR4298747,park2022proof,mossel2022second_spread}.

We now explain how the above picture suggests our main result Theorem~\ref{thm:main_informal}. Roughly speaking, the generalized threshold $\PSI_q(H)$ (Definition~\ref{d:gen.expectation.threshold})
fails to be constant over $q\in(0,1)$ if and only if the ``least likely'' subgraph $J\subseteq H$ has $|J|/|H|=q\in(0,1)$. Moreover, suppose for simplicity that $\BO p' \equiv \max_{J' \not = J}\pmmt(J')<p$.
From the ``clustering'' intuition mentioned above, we expect that whenever the \emph{null} graph $G(n,p)$ has copies of $H$, all of them should appear as extensions of much fewer copies of the ``less likely'' $J$.
%Now the ``planting trick'' from the study of random CSPs (a variant of Bayes's rule; see \cite{achlioptas2008algorithmic}) implies a similar ``clustering'' phenomenon for the \emph{planted} model. 
Now consider $\BO p'<p<\pmmt(J)$: since $p < \pmmt(J)$, the null $G(n,p)$ contains no copies of $J$, hence none of $H$. However, in the \emph{planted} model there should be a plethora of copies of $H$, all intersecting with the planted copy $\bH$ on its ``least likely'' subgraph $\bJ\subseteq\bH$ --- this is heuristically justified by the assumption that $p > p'$, and the ``least likely'' $\bJ\subseteq\bH$ is already planted. For such $p$, we expect that it will be possible to recover $\bJ$ (so we are not in a ``nothing'' phase), but it will be impossible to distinguish the true $\bH$ among all the overlapping copies (so we are not in an ``all'' phase). Our theorem establishes that this intuition is indeed valid, in fact as an equivalence statement, and that in many settings AoN can occur if and only if $\PSI_q(H)$ is roughly constant over $q$.  

Most of the above discussion is based on a heuristic picture and state-of-the-art conjectures in random graph theory. It gives the guiding intuition for this work, but we emphasize that our proof proceeds in a quite different manner, with more direct characterizations of ``all'' and ''nothing'' phases. Specifically we first establish, via a combination of the planting trick, Nishimori identity and a second moment method argument, a number of different results linking the MMSE of the planted model, with the subgraph structure of $H$ (see e.g.\ Lemmas~\ref{l:mmse.ubd} and \ref{l:mmse.lbd}). These intermediate results, allows us to apply the above intuition and obtain the AoN characterization for dense graphs $H$ (Section \ref{s:AoN.dense}), as well as the general ``nothing'' result via the spread condition (Theorem \ref{t:nothing.spread.subgraph}). For our characterization for AoN in the exponential scale, we first prove a variant of the I-MMSE relation in our setting (Lemma \ref{l:kl.divergence.derivative}) which allows to locate the AoN threshold in the exponential scale (Theorem \ref{thm:locating}). Then the planting trick alongside second moment method argument, allows us to argue again using our intuition of the previous paragraph and conclude the AoN characterization in the exponential scale (Theorem \ref{t:aon.exp}).

\section{AoN in a general Bernoulli model: definitions and key sufficient conditions}
\label{s:AoN.general}

In this section we derive general tools in a natural abstraction of the planted subgraph model,
which we term the Bernoulli inference model (Definition~\ref{d:bern.inference}),
where the graphs are replaced by more general binary vectors.
\begin{itemize}
\item In \S\ref{ss:bernoulli.model} we formally define the Bernoulli inference model,
and prove Theorem~\ref{t:bern.aon.linear} which gives a sufficient condition for AoN in this model. 

\item In \S\ref{ss:prelim.exp} we state Theorem~\ref{t:bern.aon.exp}, which gives an analogue of 
Theorem~\ref{t:bern.aon.linear} at exponential scale. \BO The exponential scale will be investigated further for the planted subgraph model in Section~\ref{s:exp}. \EH 

\item In \S\ref{ss:all} we prove the ``all'' results of Theorems \ref{t:bern.aon.linear} and \ref{t:bern.aon.exp} by a truncated first moment calculation in the planted model.

\item In \S\ref{ss:nothing} we prove the ``nothing'' results of Theorems \ref{t:bern.aon.linear} and \ref{t:bern.aon.exp} by a truncated second moment calculation in the null model.

\item In \S\ref{ss:spread} we prove Theorem~\ref{t:nothing.spread.bern}, which gives a ``nothing'' regime for the Bernoulli model in terms of the spread property. \BO As a consequence we deduce Theorem~\ref{t:nothing.spread.subgraph}, which was mentioned in the introduction as an inference version of the second Kahn--Kalai conjecture. \EH 
\end{itemize}

\subsection{General Bernoulli inference model}\label{ss:bernoulli.model} 
We begin this subsection by formally defining the general Bernoulli inference model. The main result of this subsection is Theorem~\ref{t:bern.aon.linear}, which gives a sufficient condition for AoN (at linear scale) in this model. As an application, at the end of this subsection we prove Corollary~\ref{c:planted.clique}, characterizing the occurrence of AoN in the planted clique model.

\begin{definition}[Bernoulli inference  model]
\label{d:bern.inference}
Let $N$ and $K \leq N$. Assume a \textit{uniform prior} $\bP$ on certain family $\cS$ of $M$ $K$-subsets of $[N]$, that is,
	\[\cS \subseteq \binom{[N]}{K}\,,
	\quad |\cS|=M\,,
	\]
and $\bP(S)=1/M$ for all $S\in\cS$. We assume the model is marginally symmetric, meaning that $\bP(i\in \bS)=K/N$ for all $i\in[N]$. Denote $\theta\equiv\mathbf{1}_S\in\{0,1\}^N$; with a minor abuse of notation we also write $\cS$ for the set of vectors $\theta=\mathbf{1}_S$ (for $S\in\cS$). In the \emph{Bernoulli inference model}, we first sample the (hidden) signal $\bS \sim \bP$, and denote $\theta^*\equiv\mathbf{1}_{\bS}$. We then let $\bV\subseteq[N]$ be the random subset which contains each element of $[N]$ independently with probability $p$, so that $w\equiv \mathbf{1}_{\bV}\sim\Bern(p)^{\otimes N}$. We observe $\bY\equiv\bS\cup\bV$, equivalently, $x\equiv \theta^*\vee w$. The goal is to recover $\theta^*=\mathbf{1}_{\bS}$ from $bY$. We let the planted model $\P\equiv\P_p$ denote the joint law of $(\bS,\bV,\bY)$. For comparison, we let $\Q\equiv \Q_p$ denote the null model where there is no hidden signal $\bS$, so $\bY=\bV$ and $x=w=\mathbf{1}_{\bV}\sim \Bern(p)^{\otimes N}$. 
\end{definition}

The Bernoulli model defined above
is clearly an abstraction of the
planted subgraph model (Definition~\ref{d:planted.subgr}),  with $[N]$ corresponding to the set of all available edges in $K_n$, and $\theta^*$ corresponding to the (edges of the) hidden subgraph. Note that the general Bernoulli model need not have the geometric structure of edges connected by vertices.

\BO Generalizing \eqref{e:mmse}, \EH our measure of ``recovery'' in the Bernoulli inference model (and as a consequence also for the planted subgraph model) is the \emph{minimimum mean squared error} (MMSE):
	\beq\label{e:mmse.bern}
	\mmse_N(p)
	\equiv\E_{\P_p}
	\bigg[\Big\|\theta^*-\E[\theta^*|x]\Big\|^2_2\bigg]
	= 
	\E_{\P_p} \bigg[\|\theta^*\|^2 - \Big\|\E[\theta^*|x]\Big\|^2\bigg]
	\,.\eeq
Since $\|\theta^*\|^2=K$, the MMSE must lie in $[0,K]$, so we always normalize it by $K$ in what follows. Recall that in the planted subgraph model, $\theta^*$ corresponds to the set of edges in the hidden subgraph; thus MMSE is a measure of \emph{edge recovery} rather than vertex recovery.

\BO 
It is well known that $\mmse_N(p)$ is a non-decreasing function of $p\in[0,1]$. This fact was already mentioned (and used) in the introduction, and we review the short proof here: \EH 

\begin{lemma}[monotonicity of MMSE]
\label{l:mmse.nondec}
$\mmse_N(p)$ is a non-decreasing function of $p\in[0,1]$.
\end{lemma}

\begin{proof}For $0\le p\le p'\le 1$, consider a coupling $\P$ where 
$\theta\sim\bP$,
$w$ is marginally distributed according to $\Bern(p)^{\otimes N}$, 
$w'$ is  marginally distributed according to $\Bern(p)^{\otimes N}$,
and $w \le w'$ (coordinatewise). Then, under this coupling,
with $x=\theta^*\vee w$ and $x'=\theta^*\vee w'$, we have
	\[
	\mmse_N(p)
	=\E\Big[\Var(\theta^*\,|\,x)\Big]
	=\E\Big[\Var(\theta^*\,|\,x,x')\Big]
	\le\E\Big[\Var(\theta^*\,|\,x')\Big]
	=\mmse_N(p')\,,
	\]
where the second identity is justified because if we already know $x$, then knowing $x'$ gives no additional information on $\theta^*$.
\end{proof}

\BO Given Lemma~\ref{l:mmse.nondec}, it is natural to ask in what situations the MMSE has a \emph{sharp} transition. We therefore make the following definition, generalizing Definition~\ref{d:aon.linear.subgraph} for the planted subgraph model: \EH

\begin{definition}[all-or-nothing]\label{d:aon.linear.bern}
We say that the model from Definition~\ref{d:bern.inference}
exhibits an \emph{all-or-nothing (AoN) transition} at critical probability $\pAoN=\pAoN(N)$ if
\[
\lim_{N \to \infty}  \frac{\mmse_N(p)}{K}
= \left\{
\begin{array}{ll}
1 & \textup{for all $p\ge(1+\epsilon)\pAoN$} \\
0 & \textup{for all $p\le(1-\epsilon)\pAoN$}\,,
\end{array}
\right.
\]
for any constant $\epsilon>0$. We will sometimes refer to this as ``AoN at linear scale,'' to distinguish it from the scaling of Definition~\ref{d:aon.exp} below.
\end{definition}

One of the goals of this paper is to characterize the conditions for the all-or-nothing phenomenon to hold, and if so, how to locate the threshold $\pAoN$. As mentioned above, a theme of this paper is that the location of $\pAoN$ is closely connected to first moment thresholds in the null model $\QQ_p$. For a given family of sets $\cS$ (on which the prior $\bP$ is uniformly distributed upon) the first moment threshold is the smallest value of $p$ such that the expected number of elements of $\cS$ in a sample from the null model $\Q_p$ is at least one. More formally, for $\BO \bY=\bV$ let $\bZ(\bY)$ denote the number of elements of $\cS$ that are contained in $\bY$,
	\beq\label{e:Z}
	\bZ(\bY)
	= \sum_{S\in\cS} \Ind{S\subseteq \bY}\,.
	\eeq
We define the \emph{first moment threshold}
$\pmmt$ to be the value of $p$ that satisfies
	\beq\label{def:p1M}
	1 = \E_{\Q_p} \bZ(\bY)
	= \sum_{S\in\cS} \Q_p(S\subseteq \bY)
	= M p^K\,,
	\end{equation}
that is, $\pmmt\equiv M^{-1/K}$. We use this value of $p$ to specify the following \emph{growth condition} on the prior, which controls from above the probability that two independent draws of the prior overlap in a specific number of elements:
	\begin{equation}\label{e:growth}
	\limsup_{N\to\infty} \sup_{0\le\ell\le K}
	\frac1K 
	\log \bigg(
	\frac{\bP^{\otimes 2}( |\bS \cap \bS'|= \ell)}{(\pmmt)^\ell}
	 \bigg)
	 \leq 0\,,
\end{equation} where $\bS$ and $\bS'$ are independent draws from $\bP$ (the uniform measure over $\cS$). We highlight that in the language of~\cite{AoNtensor}, the growth condition~\eqref{e:growth} is a bound on the \emph{overlap rate function} of the prior $\bP$; that paper showed that a similar condition implies the existence of an all-or-nothing threshold for sparse estimation problems with gaussian noise.

%\BH Note that \eqref{e:Z} and \eqref{def:p1M}
%generalize our earlier definitions
%\eqref{e:Z.graph} and \eqref{e:graph.first.mmt.threshold} for planted subgraph models. \EH
%We only consider the case $K=O(\log M)$, since otherwise $\pmmt=1-o(1)$. 

% The second condition is
% \begin{equation}\label{e:weak.growth}
% \limsup_{N\to\infty} \sup_{0\le\ell\le K}
% \frac{1}{\log M}\log \bigg(
% 	\frac{\bP^{\otimes 2}( |\bS \cap \bS'|= \ell)}
% 		{(\pmmt)^\ell}
% 	\bigg)
% 	 \leq 0\,,
% \end{equation} where again $\bS$ and $\bS'$ are independent draws from $\bP$. Clearly, condition~\eqref{e:growth} implies
% condition~\eqref{e:weak.growth} under our assumption $K=O(\log M)$. 
% \BO In \S\ref{ss:spread} we discuss connections between \eqref{e:growth}, \eqref{e:weak.growth}, and the well-known notion of ``spread.'' \EH

\begin{theorem}[AoN for Bernoulli inference model under growth condition]
\label{t:bern.aon.linear}
Suppose the prior $\bP$ satisfies condition~\eqref{e:growth}, and moreover that $K=K(N) \to\infty$. Then the associated Bernoulli inference model
(Definition~\ref{d:bern.inference}) has AoN at $\pAoN=\pmmt$ as defined by \eqref{def:p1M}: that is,
\[
\lim_{N \to \infty}  \frac{\mmse_N(p)}{K} 
= \left\{
\begin{array}{ll}
1 & \textup{if $p\ge(1+\epsilon)\pmmt$,} \\
0 & \textup{if $p\le(1-\epsilon)\pmmt$,}
\end{array}
\right.\]
for any constant $\epsilon>0$.
\end{theorem}

From Theorem~\ref{t:bern.aon.linear} it is straightforward to deduce the following characterization of AoN in the planted clique model:

%Notice that Theorem~\ref{t:bern.aon.linear} implies in a straightforward manner that the planted clique model exhibits AoN if and only if the planted clique's size diverges. In case AoN appears in the planted clique model, it appears at the first moment threshold.

\begin{corollary}\label{c:planted.clique}
Consider the model of a clique $H$ on $k$ vertices planted in $G(n,p)$.
\begin{enumerate}[(a)]
    \item If $k=k_n \rightarrow +\infty$  then the model exhibits AoN at
    \[
    \pAoN(H) = \pmmt(H) = \binom{n}{k}^{-1/\binom{k}{2}}\,.
    \]
    \item If $k=k_n=O(1)$ then the model does not exhibit AoN at any value of $p$.
\end{enumerate}
\end{corollary}

\BO For the proof of Corollary~\ref{c:planted.clique} we 
review the so-called ``Nishimori identity,'' \EH which refers to the fact that the pair $(\theta^*,\theta')$ (the original signal, together with one sample from the posterior distribution) is equidistributed as the pair $(\theta',\theta'')$ (two independent samples from the posterior distribution). This is a basic consequence of Bayes's rule:
	\begin{align}\nonumber
	\P(\theta^*=\theta_1,\theta'=\theta_2)
	&=\bP(\theta^*=\theta_1)
	\sum_x
	\P(x\,|\,\theta^*=\theta_1)
	\P(\theta'=\theta_2\,|\,x) \\
	&= \sum_x \P(x)
	\P(\theta^*=\theta_1\,|\,x)
	\P(\theta'=\theta_2\,|\,x)
	= \P(\theta''=\theta_1,\theta'=\theta_2)\,,
	\label{e:nishimori}
	\end{align}
and the pair $(\theta',\theta'')$ is clearly exchangeable.

\begin{proof}[Proof of Corollary~\ref{c:planted.clique}]
For part (a), by Theorem~\ref{t:bern.aon.linear} it suffices to check the growth condition \eqref{e:growth}. The quantities of Definition~\ref{d:bern.inference} in the context of the planted clique model are
    \[
    M = \binom{n}{k}\,,\quad
    K = \binom{k}{2}\,.
    \]
Let $\bS$ be a sample from the prior $\bP$, that is, $\bS$ is uniformly random among all $k$-cliques contained in the complete graph $K_n$. If $\bS'$ is an independent copy of $\bS$, then
    \begin{align*}
    &\frac{\bP^{\otimes2}(
    |\bS\cap \bS'|\ge\ell)}{\pmmt(H)^\ell}
    \le
    \sum_{0\le t\le k}
    \mathbf{1}\bigg\{ \binom{t}{2} \ge\ell\bigg\}
    \binom{k}{t}
    \binom{n-k}{k-t}
    \binom{n}{k}^{-1 + \ell/\binom{k}{2} } \\
    &\qquad\le \exp(o(k^2))
    \sum_{0\le t\le k}
    \mathbf{1}\bigg\{ \binom{t}{2} \ge\ell\bigg\}
    n^{k-t} (n^k)^{-1+\ell/\binom{k}{2}}\\    &\qquad\le\frac{\exp(o(k^2))}{
    \exp\{ \Theta[\ell^{1/2} - \ell/k^2] \log n\}
    }
    \le\exp(o(k^2))\,.
    \end{align*}
This verifies condition~\eqref{e:growth}, and the claim follows.

For part (b) we focus on the regime $p\asymp\pcrit(H)$ with $\pcrit(H)$. It suffices to show that for any such $p$, the normalized MMSE does not tend to zero or one. In light of \eqref{e:nishimori}, it suffices to show that if $\bH',\bH''$ are two samples from the posterior distribution given $\bY=\bH\cup\bG$, then the expected edge-overlap between $\bH'$ and $\bH''$, normalized by $|H|$, is not tending to zero or one. To this end, recall that $\bZ=\bZ(\bY)$ denotes the total number of $k$-cliques contained in $\bY$. We have
	\[
	\E_{\P_p}\bZ
	= \sum_{w=0}^k
	\binom{k}{w}
	\binom{n-k}{k-w}
	p^{\binom{k}{2}-\binom{w}{2}}
	\]
where $w$ is the number of vertices shared with the planted clique. Recalling $k\asymp1$, we have
	\[
	\E_{\P_p}\bZ
	\asymp
	\sum_{w=0}^k n^{k-w}
	\frac{(n^k)^{\binom{w}{2}/\binom{k}{2}}}{n^k}
	\asymp
	\sum_{w=0}^k \bigg(\frac{1}{n^w}\bigg)^{1- \frac{w-1}{k-1}}
	\le O(1)\,.
	\]
This shows that $\bZ$ is a stochastically bounded random variable. For AoN to hold, the only possibility is that $\bZ=1$ with high probability. However, for such values of $p$ it follows from standard results that the graph $\bG\setminus\bH\sim G(n-k,p)$ contains a clique with nonnegligible probability, which implies $\bZ\ge2$ with nonnegligible probability. This shows that the normalized MMSE does not converge to zero or one in this regime of $p$, so AoN does not occur.
\end{proof}

\subsection{Results for exponential scale}
\label{ss:prelim.exp}

In this subsection we state another result concerning AoN in the Bernoulli model, Theorem~\ref{t:bern.aon.exp} below. It is very similar to Theorem~\ref{t:bern.aon.linear}, but with AoN happening at a larger scale (Definition~\ref{d:aon.exp}), and under a weaker condition. The proofs for Theorems~\ref{t:bern.aon.linear} and \ref{t:bern.aon.exp} are very similar, and we present both in this section. %We will sometimes refer to the situation of Definition~\ref{d:aon.linear.bern} as ``\emph{AoN at linear scale},'' to distinguish it from the following case:

\begin{definition}[AoN at exponential scale]
\label{d:aon.exp}
We say that the model from Definition~\ref{d:bern.inference}
 exhibits an \emph{all-or-nothing transition at the exponential scale} 
 with critical probability $\pAoN=\pAoN(N)$ if
\[
\lim_{n \to \infty}  \frac{\mmse_N(p)}{K} = \left\{
\begin{array}{ll}
1 & \textup{for all $p\ge(\pAoN)^{1-\epsilon}$} \\
0 & \textup{for all $p\le(\pAoN)^{1+\epsilon}$}\,,
\end{array}
\right.\]
for any constant $\epsilon>0$, where we further require
$\pAoN$ to stay bounded away from one for the definition to be meaningful.
\end{definition}

We note that if $p$ is bounded away from one then 
	\[
	(1+o_n(1)) p
	= \exp\bigg( o_n(1) + \log p\bigg)
	= \exp\bigg( \Big(1+o_n(1)\Big) \log p\bigg)\,,
	\]
so AoN at the linear scale at a threshold bounded away from one implies AoN at the exponential scale at that same threshold. The converse is clearly false. The following is a variant of the growth condition \eqref{e:growth}:
\beq\label{e:weak.growth}
\limsup_{N\to\infty} \sup_{0\le\ell\le K}
\frac{1}{\log M}\log \bigg(
	\frac{\bP^{\otimes 2}( |\bS \cap \bS'|= \ell)}
		{(\pcrit)^\ell}
	\bigg)
	 \leq 0\,,
\eeq
We assume throughout that $K=O(\log M)$, which ensures that $\pmmt$ is bounded away from one. Under this assumption, condition~\eqref{e:growth} implies condition~\eqref{e:weak.growth}.

\begin{theorem}[AoN at exponential scale under \eqref{e:weak.growth}]
\label{t:bern.aon.exp}
Suppose that the prior $\bP$ satisfies \eqref{e:weak.growth}, and moreover that $K\to\infty$. Then the associated model has AoN at the exponential scale at $\pAoN=\pcrit$ as defined by \eqref{def:p1M}: that is,
\[
\lim_{N \to \infty}  \frac{\mmse_N(p)}{K} 
= \left\{
\begin{array}{ll}
1 & \textup{if $p\ge (\pcrit)^{1-\epsilon}$} \\
0 & \textup{if $p\le (\pcrit)^{1+\epsilon}$}\,.
\end{array}
\right.\]
for any constant $\epsilon>0$.
\end{theorem}

\subsection{Proofs for ``all'' regime by truncated first moment in planted model}\label{ss:all}
In this subsection we prove the positive (``all'') results of 
Theorems~\ref{t:bern.aon.linear} and \ref{t:bern.aon.exp}, showing in both cases,
if $p$ is sufficiently below $\pmmt$, then the (normalized) MMSE tends to zero.

%\begin{theorem}
%Suppose that the prior satisfies \eqref{e:growth} and that $K \rightarrow +\infty.$ Then if for any $\epsilon>0$ it holds $p \leq (1-\epsilon)\pmmt$ then \begin{equation}\label{all}
%\lim_{N \to \infty}  \frac{1}{K} \mathrm{MMSE}_N(K, \cS,p)  = 0 
%\end{equation}
%\end{theorem}

\begin{lemma}[MMSE upper bound]\label{l:mmse.ubd}
	{\BG For any $\delta > 0$},
if the prior $\bP$ satisfies
	\beq\label{e:mmse.ubd.condition}
	\sum_{\ell \le (\BG{1-\delta)}K}
	\bP^{\otimes 2}\Big( |\bS\cap \bS'|=\ell\Big) Mp^{K-\ell}
	= o_N(1)
	\eeq
then $\mmse_N(p)/K \le \BG{\delta} + o_N(1)$.
\end{lemma}

\begin{proof}
Recall that under the planted model $\P=\P_p$, the observation is $\bY=\bS\cup\bV$ where $\bV$ contains each element of $[N]$ independently with probability $p$. We will argue that, with high probability, \emph{any} element $S'\in\cS$ satisfying $S'\subseteq\bY$ must have overlap at least $K(1-\delta)$ with the planted set $\bS$. Since the Bayes estimator $\Btheta=\E(\theta^*\,|\,x)$ is the average of $\mathbf{1}_{S'}$ over all such $S'$, it follows by linearity that $(\theta^*,\Btheta) \ge K({\BG 1 - \delta}-o_N(1))$ with high probability, and therefore we have
	\[
	\frac{\mmse}{K}
	\stackrel{\eqref{e:mmse}}{=} 
	\frac{1}{K} \E_{\P_p} \bigg[\|\theta^*\|^2 - 
		\Big(\theta^* ,\Btheta\Big)\bigg]
	\le {\BG \delta}+o_N(1)\,,
	\]
as claimed. Thus, using Markov's inequality, the result will follow once
we show
	\beq\label{e:truncated.first.mmt}
	\sum_{\ell \le K\delta} \E_{\P} \bZ_\ell(\bS,\bY) = o_N(1)\,,
	\eeq
where $\bZ_\ell(\bS,\bY)$ denotes the number of subsets $S'\in\cS$ that are contained in $\bY$ and have overlap $\ell$ with the hidden signal $\bS$. Then note that for any $\ell$ we have	
	\beq\label{e:all.bound.A}
	\E_{\P} \bZ_\ell(\bS,\bY)
	{\BG = \E_{\P}} \bigg|\Big\{S'\in\cS : |\bS\cap S'|=\ell\Big\}\bigg| p^{K-\ell}
	= \bP^{\otimes 2}\Big( |\bS\cap \bS'|=\ell\Big) Mp^{K-\ell}
	\,.
	\eeq
In the above we have used that if we condition on $\bS$ and consider $S'$ with $|\bS\cap S'|=\ell$, then the probability $\BG S'$ is contained in $\bY=\bS\cup\bV$ is the same as the probability that $S'\setminus \bS$ is contained in $\bV$, which is $p^{K-\ell}$.
It follows that \eqref{e:mmse.ubd.condition} implies \eqref{e:truncated.first.mmt}{\BG, which proves the claim.}
%Thus, the estimator which exhaustively goes over all $S'\in\cS$ and returns any $S'$ satisfying $S'\subseteq\bY$ succeeds with high probability to return an $S'$ that has at least $\delta$-overlap with $\bS$.
\end{proof}

\begin{proof}[Proof of Theorem~\ref{t:bern.aon.linear} ``all'' result]
By Lemma~\ref{l:mmse.ubd}, it suffices to check that  the condition~\eqref{e:mmse.ubd.condition} holds for
any $p \le (1-\epsilon)\pmmt$, and for
 any $\delta>0$. We first use the assumption on $p$ to bound
	\beq\label{e:all.bound.B}
	Mp^{K-\ell}
	\le M (\pmmt)^{K-\ell}(1-\epsilon)^{K-\ell}
	\stackrel{\eqref{def:p1M}}{=}
	M\bigg(\frac{1}{M}\bigg)^{(K-\ell)/K}
	(1-\epsilon)^{K-\ell}
	= M^{\ell/K}(1-\epsilon)^{K-\ell}\,.
	\eeq
It follows by combining with the growth condition \eqref{e:growth} that
for all $\ell\le K (1-\delta)$,
	\begin{align}\nonumber 
	&\bP^{\otimes 2}\Big( |\bS\cap \bS'|=\ell\Big) Mp^{K-\ell}
	\le
	\exp\bigg\{
	-\delta \epsilon K + \log \bigg[
	\bP^{\otimes 2}\Big( |\bS \cap \bS'|= \ell\Big) M^{\ell/K} \bigg]
	\bigg\}\\
	&\qquad\stackrel{\eqref{e:growth}}{\le} 
	\exp\Big\{ -\delta \epsilon K + o(K)\Big\}\,.
	\label{e:all.bound.C}
	\end{align}
Since $K\to\infty$, it follows by summing over $\ell \le K(1-\delta)$ that
condition~\eqref{e:mmse.ubd.condition} holds. The claim follows.
\end{proof}
%Thus, the estimator which exhaustively goes over all $S'\in\cS$ and returns any $S'$ satisfying $S'\subseteq\bY$ succeeds with high probability to return an $S'$ that has large overlap with $\bS$.

\begin{proof}[Proof of Theorem~\ref{t:bern.aon.exp} ``all'' result]
Again, by Lemma~\ref{l:mmse.ubd}, it suffices to check that condition~\eqref{e:mmse.ubd.condition} holds 
for any $p \le (\pmmt)^{1+\epsilon}$, and 
for any $\delta>0$.
%We follow the same structure as the preceding proof. 
%For $\P=\P_p$, we first bound $\E_{\P} \bZ_\ell(\bS,\bY)$ as in \eqref{e:all.bound.A}. 
Using the assumption on $p$,
instead of \eqref{e:all.bound.B} we have
	\[
	Mp^{K-\ell}
	\le M(\pmmt)^{(K-\ell)(1+\epsilon)}
	\stackrel{\eqref{def:p1M}}{=}
	M\bigg(\frac{1}{M}\bigg)^{(1-\ell/K)(1+\epsilon)}
	= \frac{M^{\ell/K}}{M^{\epsilon(1-\ell/K)}}\,.
	\]
Combining with the growth condition \eqref{e:weak.growth}, instead of \eqref{e:all.bound.C} we have
	\[
	\bP^{\otimes 2}\Big( |\bS\cap \bS'|=\ell\Big) Mp^{K-\ell}
	\le 
	\frac{\exp(o(\log M))}{ M^{\epsilon(1-\ell/K)}}
	\le
	\frac{\exp(o(\log M))}{ M^{\epsilon\delta}}
	\]
for each $\ell\le K(1-\delta)$. Since we assumed $K\to\infty$ and $\BH K=O(\log M)$, we have $\log K\ll K \le O(\log M)$. It follows by summing over $\ell \le K(1-\delta)$ that
condition~\eqref{e:mmse.ubd.condition} holds, and this proves the claim.\end{proof}

\subsection{Proofs for ``nothing'' regime by truncated second moment in null model}\label{ss:nothing}
In this subsection we prove the negative (``nothing'') results of 
Theorems~\ref{t:bern.aon.linear} and \ref{t:bern.aon.exp}, showing that in both cases, if $p$ is sufficiently above $\pmmt$, then the scaled MMSE tends to one.

Recall that in the planted model $\P=\P_p$, there is a hidden signal $\bS\sim\bP$ and we observe $\bY=\bS\cup\bV$ where $\bV$ is an independent $p$-biased random set. We now compare this with the null model $\Q=\Q_p$, where there is no hidden signal and we observe $\bY=\bV$. Note that
	\begin{align}\nonumber
	\P(\bS=S,\bY=Y)
	&=\frac{\bP(S)\Ind{S\subseteq Y}\Q(\bY=Y)}{p^{|S|}}
	= \frac{\Ind{S\subseteq Y}\Q(\bY=Y)}{Mp^K}\\
	&= \frac{\Ind{S\subseteq Y}\Q(\bY=Y)}{\E_{\Q} \bZ(\bY)}\,,
	\label{e:P.vs.Q}
	\end{align}
with $\bZ(Y)$ as in \eqref{e:Z}. It follows that the marginal law of $\bY$ under the planted model is given by
	\beq\label{e:radon.nikodym}
	\P(\bY=Y)
	= \frac{\bZ(\bY)}{\E_{\Q} \bZ(\bY)}
	\Q(\bY=Y)\,.
	\eeq
That is, the Radon--Nikodym derivative between $\P$ and $\Q$ is given by the ratio of $\bZ(\bY)$ to $\E_{\Q} \bZ(\bY)$. An immediate consequence is that this ratio is unlikely to be small under $\P$:

\begin{lemma}\label{lem:planting_first}
For any $\epsilon>0$ we have
$\P( \bZ(\bY) \leq \epsilon \E_{\Q} \bZ(\bY) ) \leq \epsilon$.
\end{lemma}

\begin{proof}
It follows from \eqref{e:radon.nikodym} that
	\[
	\P\Big( \bZ(\bY) \leq \epsilon \E_{\Q} \bZ(\bY) \Big)
	= \E_{\Q}\bigg[ 
	\mathbf{1}\Big\{\bZ(\bY) \leq \epsilon \E_{\Q} \bZ(\bY)\Big\}
	\frac{\bZ(\bY)}{\E_{\Q} \bZ(\bY)}
	\bigg]
	\le \epsilon
	\Q\Big( \bZ(\bY) \leq \epsilon \E_{\Q} \bZ(\bY) \Big)
	\le\epsilon\,,
	\]
as claimed.
\end{proof}

Recall from above that $\bZ_\ell(\bS,\bY)$ denotes the number of subsets $S'\in\cS$ that are contained in $\bY$ and have overlap $\ell$ with the hidden signal $\bS$. We also let $\bZ^2(\ell,\bY)$ denote the number of pairs $(S',S'')\in\cS^2$ with $|S'\cap S''|=\ell$ and $S'\subseteq\bY$, $S''\subseteq\bY$. The following is a consequence of \eqref{e:P.vs.Q}:

\begin{lemma}\label{lem:change_of_measure}
For any $0\le\ell\le K$, we have
\[\E_{\P} \bZ_\ell(\bS,\bY)
=\frac{\E_{\Q}[ \bZ^2(\ell,\bY) ]}{\E_{\Q}\bZ(\bY)}
\,.\]
\end{lemma}

\begin{proof}
It follows from \eqref{e:P.vs.Q} that
	\begin{align*}
	\E_{\P} \bZ_\ell(\bS,\bY)
	&= 
	\sum_{S,Y}
	\frac{\Ind{S\subseteq Y}
	\Q(\bY=Y)}{\E_{\Q} \bZ(\bY)}
	\bZ_\ell(S,Y)\\
	&=\sum_Y
	\frac{\Q(\bY=Y)}{\E_{\Q} \bZ(\bY)}
	\sum_S \Ind{S\subseteq Y}
	\sum_{S'}\Ind{S'\subseteq Y}
	\Ind{|S\cap S'|=\ell}
	= \frac{\E_{\Q}[ \bZ^2(\ell,\bY) ]}{\E_{\Q}\bZ(\bY)}\,,
	\end{align*}
as claimed.
\end{proof}

%We unify the proofs as they differ only in the last step
%\begin{theorem}\label{thm:nothing:linear}
%Suppose that the prior satisfies \eqref{e:growth} and that $K \rightarrow +\infty.$ Then if for any $\epsilon>0$ it holds $p \geq (1+\epsilon)\pmmt$ then \begin{equation}\label{all}
%\lim_{N \to \infty}  \frac{1}{K} \mathrm{MMSE}_N(K, \cS,p)  = 1 
%\end{equation}
%\end{theorem}
%
%and
%
%\begin{theorem}
%Suppose that the prior satisfies \eqref{e:weak.growth} and that $\log M \rightarrow +\infty$ and $\log K=o(\log M)$. Then if for any $\epsilon>0$ it holds $p \geq \pmmt^{1-\epsilon}$ then \begin{equation}\label{all}
%\lim_{N \to \infty}  \frac{1}{K} \mathrm{MMSE}_N(K, \cS,p)  = 1 
%\end{equation}
%\end{theorem}
%
%We unify the two proofs in one as the differ only in the last step.

The next result gives a counterpart to Lemma~\ref{l:mmse.ubd}:

\begin{lemma}[MMSE lower bound]
\label{l:mmse.lbd}
If the prior $\bP$ satisfies 
	\beq\label{e:mmse.lbd.condition}
	\sum_{\ell \ge K \delta}
	\frac{\bP^{\otimes 2}(|\bS \cap \bS'| = \ell)}{p^\ell} =o_N(1)\,,
	\eeq
then $\mmse_N(p)/K \ge 1- \delta - o_N(1)$.
\end{lemma}

\begin{proof}
Let $\bS'$ denote a sample from the posterior distribution of $\bS$ given $\bY$. We will show that, for $\P=\P_p$, we have the bound
	\[
	\P \bigg(\frac{|\bS\cap\bS'|}{K} \ge\delta\bigg) = o_N(1)\,.
	\]
Since the Bayes estimator $\Btheta$ is the expectation of $\theta'=\mathbf{1}_{\bS'}$ given $\bY$, it follows that
	\[\frac{\mmse}{K}
	\stackrel{\eqref{e:mmse}}{=} 
	\frac{1}{K} \E_{\P_p} \bigg[\|\theta^*\|^2 - 
		\Big(\theta^* ,\Btheta\Big)\bigg]
	\ge
	1-\delta + o_N(1)\,.
	\]
Now note that we can apply Lemmas~\ref{lem:planting_first} and \ref{lem:change_of_measure}
to bound, for any $\epsilon>0$,
	\begin{align*}
	&\P \bigg(\frac{|\bS\cap\bS'|}{K} \ge\delta\bigg)
	= \E_{\P} \bigg[\frac{1}{\bZ(\bY)}
	\sum_{\ell\ge K\delta} \bZ_\ell(\bS,\bY)\bigg]
	\le
	\epsilon 
	+\frac1{\epsilon\E_{\Q}\bZ(\bY)}
	\E_{\P} \bigg[
	\sum_{\ell\ge K\delta} \bZ_\ell(\bS,\bY)\bigg] \\
	&\qquad =\epsilon 
	+\frac1{\epsilon (\E_{\Q}\bZ(\bY))^2}
	\E_{\Q} \bigg[
	\sum_{\ell\ge K\delta} \bZ^2(\ell,\bY)\bigg]
	\equiv \epsilon +
	 \frac{p(\delta)}{\epsilon}
	\end{align*}
Since $\epsilon$ is arbitrary, it suffices to show
$p(\delta)=o_N(1)$. Then note that 
	\[p_\ell(\delta)\equiv
  \frac{\E_{\Q} \bZ^2(\ell,\bY)}{(\E_{\Q}\bZ(\bY))^2}
  \BM = \EH 
  \frac{ M^2\bP^{\otimes 2}(|\bS \cap \bS'| = \ell)p^{2K-\ell}}{(Mp^K)^2} 
 = \frac{\bP^{\otimes 2}(|\bS \cap \bS'| = \ell)}{p^\ell}\,.\]
Summing over $\ell\ge K\delta$ gives the left-hand side of \eqref{e:mmse.lbd.condition}, and proves the claim.
\end{proof}

\begin{proof}[\hypertarget{thm:nothing:linear}{Proof of Theorem~\ref{t:bern.aon.linear} ``nothing'' result}]
By Lemma~\ref{l:mmse.lbd}, it suffices
to check condition~\eqref{e:mmse.lbd.condition}
for $p \ge (1+\epsilon)\pmmt$ and any $\delta>0$.
We combine the assumption on $p$ with
condition~\eqref{e:growth} to bound, for 
$\ell\ge K\delta$,
\begin{align*} 
p_\ell(\delta)
&\le
\frac{\bP^{\otimes 2}(|\bS \cap \bS'| = \ell)}{(1+\epsilon)^\ell (\pmmt)^\ell} 
\stackrel{\eqref{def:p1M}}{=}
\frac{\bP^{\otimes 2}(|\bS \cap \bS'| = \ell) M^{\ell/K}}{(1+\epsilon)^\ell} 
\\
& \le \exp\bigg\{
\log\Big[\bP^{\otimes 2}(|\bS \cap \bS'| = \ell) M^{\ell/K}\Big]
-K\delta{\BG \log(1+\epsilon)}\bigg\} 
\stackrel{\eqref{e:growth}}{\le} \exp\Big\{ o(K)-\delta K {\BG \log(1+\epsilon)}\Big\}\,.
\end{align*} Since $K\rightarrow \infty$, we can sum over $\ell\ge K\delta$ to conclude that $p(\delta)=o_N(1)$, which concludes the proof.
\end{proof}

\begin{proof}[Proof of Theorem~\ref{t:bern.aon.exp} ``nothing'' result]
Again, Lemma~\ref{l:mmse.lbd}, it suffices
to check condition~\eqref{e:mmse.lbd.condition}
for $p \ge (\pmmt)^{1-\epsilon}$ and any $\delta>0$.
We combine the assumption on $p$ with
condition~\eqref{e:weak.growth} to bound, for 
$\ell\ge K\delta$,
\begin{align*} 
p_\ell(\delta)
&\le
\frac{\bP^{\otimes 2}(|\bS \cap \bS'| = \ell)}{ (\pmmt)^{\ell(1-\epsilon)}} 
 \stackrel{\eqref{def:p1M}}{=}
\frac{\bP^{\otimes 2}(|\bS \cap \bS'| = \ell)M^{\ell/K}
}{ M^{\ell\epsilon/K} } 
\\
& \leq \exp\bigg\{ \log \Big[
\bP^{\otimes 2}(|\bS \cap \bS'| = \ell)M^{\ell/K}\Big]
- \delta\epsilon \log M\bigg\} 
\stackrel{\eqref{e:weak.growth}}{\le} \exp\Big\{ o(\log M)-\epsilon \delta \log M\Big\}
\end{align*}
Since $\log M \rightarrow \infty$ and $\BH \log K=o( \log M)$,
we can sum over $\ell\ge K\delta$ to conclude that $p(\delta)=o_N(1)$, which concludes the proof.
\end{proof}

\newcommand{\pp}{\bar{p}}

\subsection{Connections to spread: second Kahn-Kalai conjecture in inference}\label{ss:spread} 

% In recent years, the following ``spread condition'' has played a key role in a series of important results in probabilistic combinatorics  \iz{Add refs}

\BM The main result of this subsection is Theorem~\ref{t:nothing.spread.subgraph}, which can be thought of as a version of the second Kahn--Kalai in the inference setting. \EH

\begin{definition}[spread condition]\label{d:spread}
Let $\bP$ be a probability measure over a family $\cS$ of subsets of $[N]$. For $\pp<1$, we say that $\bP$ is \emph{\BM $\pp$-spread} if for any subset $A \subseteq [N]$ we have
\begin{align}\label{e:spread} 
\bP(A \subseteq \bS) \le \pp^{|A|}\,,
\end{align}
where $\bS$ is a random sample from $\bP$. Note that it suffices to check the condition for all $A\ne\varnothing$.
 \end{definition}Intuitively the growth condition \eqref{e:growth} ask for the prior $\bP$ to satisfy an approximate version of the spread condition, a celebrated condition in probabilistic combinatorics \cite{sunflower_annals,MR4298747,Rao_sunflower,Tao_sunflower,bell2021note,Hu_sunflower,stoeckl_sunflower,mossel2022second_spread}. One can wonder if the well-studied spread condition by itself also implies a version of the all-or-nothing phenomenon. In this work, we establish that any prior satisfying the following relaxation of the spread condition does satisfy ``nothing'' for a wide range of value of $p$.

\begin{definition}[\BH generalized \EH spread condition]\label{d:delta-spread}
Let $\bP$ be a probability measure over a family $\cS$ of $\BM K$-subsets of $[N]$. For $\delta>0$, we say that $\bP$ is \emph{\BM $(\pp,\delta)$-spread} if we have
\beq\label{e:delta-spread} 
\bP(A \subseteq \bS) \leq \pp^{|A|}
\eeq
for any $A\subseteq[N]$ satisfying $|A|\ge K\delta$.
\end{definition}

We have the following theorem which establishes a ``nothing'' regime for the Bernoulli inference model in terms of the generalized spread condition:

\begin{theorem}[``nothing'' under the generalized spread condition]
\label{t:nothing.spread.bern}
In the Bernoulli inference model (Definition~\ref{d:bern.inference}),
suppose for every $\delta>0$ that $\bP$ 
is $(\delta,\pp)$-spread for some $\pp=\pp_N(\delta)$.
We then have
    \[
    \lim_{N \to \infty}  \frac{\mmse_N(p)}{K} 
    = 1\]
for any $p$ satisfying \BM $p \ge 3\pp_N(\delta) /\delta$ for all $\delta>0$. \EH
\end{theorem}

\begin{proof}
By Lemma~\ref{l:mmse.lbd}, it suffices
to check condition~\eqref{e:mmse.lbd.condition}
for any constant $\delta>0$. Since $\bP$ is supported on $K$-subsets, for all $\ell\ge K\delta$ we have
    \begin{align*}
    &\bP^{\otimes 2}\Big(|\bS \cap \bS'|=\ell\Big)
    = \mathbf{E}\bigg[
    \bP\Big(|\bS \cap \bS'|=\ell \,\Big|\, \bS\Big)
    \bigg]
    = \mathbf{E}\bigg[
    \sum_{\substack{A\subseteq \bS,
        |A|=\ell}}
        \bP(A\subseteq \bS')\bigg]\\
    &\qquad\stackrel{\eqref{e:delta-spread}}{\le}
    \binom{K}{\ell} \pp^\ell
    \le \bigg(\frac{Ke \pp}{\ell} \bigg)^\ell
    \le \bigg(\frac{e \pp}{\delta} \bigg)^\ell\,.
    \end{align*}
Therefore condition~\eqref{e:mmse.lbd.condition} holds for 
$p\ge 3\pp/\delta$. This completes the proof.
\end{proof}

\BM The following gives an inference version of the second Kahn--Kalai conjecture, which was mentioned in the discussion of \S\ref{ss:beyond}: \EH

\begin{theorem}[``nothing'' for arbitrary graphs]\label{t:nothing.spread.subgraph}
Let $H=H_n$ be an arbitrary graph. The model of $H_n$ planted in $G(n,p)$ is in the ``nothing'' regime when
$p=p_n$ satisfies \BM 
    \[
    \liminf_{q\downarrow0}
    \liminf_{n\to\infty}
    \frac{p_n}{\PSI_q(H) / q}
    \ge3\,,
    \]
for $\PSI_q$ as defined by \eqref{e:psi.q}. \EH
\end{theorem}

\begin{proof}
Given Theorem~\ref{t:nothing.spread.bern}, it suffices to check that $\bP$, the uniform measure over copies of $H$ in $K_n$, satisfies the generalized spread condition with the appropriate parameters. For any nonempty subgraph $J\subseteq H$, let $\pi_J$ denote the uniform measure over copies of $J$ in $K_n$, and let $\bm{J}$ denote a sample from $\pi_J$. Let $J_0,H_0$ be arbitrary fixed copies of $J,H$ in $K_n$. Let $M_{J,H}$ denote the number of copies of $J$ in $H$. It follows by symmetry that    
\[\bP(J_0\subseteq\bH)
=\pi_J(\bm{J}\subseteq H_0)
= \frac{M_{J,H}}{M_J}
\le \frac{1}{M_J}\binom{|H|}{|J|}
\le \frac{1}{M_J} \bigg(\frac{|H|e}{|J|}\bigg)^{|J|}
\stackrel{\eqref{e:graph.first.mmt.threshold}}{=}
\bigg(\frac{|H|e \pmmt(J)}{|J|}\bigg)^{|J|}\,.
\]
If $|J|\ge |K|q$
and $\BM p \ge e\pmmt(J)/q$ then we obtain
    \[
    \bP(J_0\subseteq\bH)
    \le 
    \bigg(\frac{e \pmmt(J)}{q}\bigg)^{|J|}
    \le p^{|J|}\,,
    \]
which is the generalized spread condition. The claim follows.\end{proof}

\section{Characterization of AoN at linear scale for sufficiently dense graphs}
\label{s:AoN.dense}

The main purpose of the current section is to prove Theorem~\ref{t:dense.linear.aon}, which characterizes AoN at linear scale for dense graphs under mild conditions. The main message of Theorem~\ref{t:dense.linear.aon} is that for dense graphs $H$, AoN at linear scale can occur if and only if the graph is almost balanced, i.e., it is almost the case that $H$ is the densest subgraph of the full graph. With this in mind, the current section is organized as follows:
\begin{itemize}
\item In \S\ref{ss:aon.lin.forward} we prove the forward direction of Theorem~\ref{t:dense.linear.aon} by showing that for dense graphs that are delocalized (Definition~\ref{d:delocalized}), AoN implies the almost-balanced condition~\eqref{e:almost.bal}.

\item In \S\ref{ss:aon.lin.reverse} we prove the reverse direction of Theorem~\ref{t:dense.linear.aon} by showing that for delocalized dense graphs, condition~\eqref{e:almost.bal} implies AoN.

\item In \S\ref{ss:sparse} we prove Theorem~\ref{t:aon.sparse} which
gives sufficient conditions for AoN at linear scale for graphs
without a density restriction, but only provided the number of vertices and edges is very small. The result illustrates that we have limited understanding of AoN for sparse planted subgraphs, and we leave this as an interesting direction for future research.
\end{itemize}

We begin with some preliminaries.
Throughout the remainder of this paper, we specialize from the abstract Bernoulli inference model (Definition~\ref{d:bern.inference}) to
the planted subgraph model
(Definition~\ref{d:planted.subgr}). Recall that the explicit correspondence between Definitions \ref{d:planted.subgr} and \ref{d:bern.inference} is as follows: $[N]$ is the set of all edges in the complete graph $K_n$, $\cS$ is the set of all copies of $H$ insides $K_n$, and $M=|\cS|$ is the total number of distinct copies of $H$ in $K_n$:
	\beq\label{e:M.H}
	M \equiv M_H \equiv M_{H,K_n}
	= \frac{n! / (n-v(H))!}{|\Aut(H)|} = \frac{(n)_{v(H)}}{|\Aut(H)|}
	= \binom{n}{v(H)} \Psi_H
	\,,
	\eeq
where $\Aut(H)$ denotes the automorphism group of $H$,
and $\Psi_H$ denotes the number of distinct copies of $H$ inside the complete graph on $v(H)$ vertices. We have the trivial bounds
	\beq\label{e:trivial.bounds.on.M}
	\binom{n}{v(H)}
	\le M_H \le (n)_{v(H)} \le n^{v(H)}\,.
	\eeq
Recall that $\bP$ denotes the uniform probability measure over $\cS$. Under the planted model $\P_p$ we observed $\bY=\bH\cup\bG$ where $\bH\sim\bP$ and $\bG$ is an independent sample from $G(n,p)$. The first moment threshold $\pmmt(H)$ is defined by \eqref{e:graph.first.mmt.threshold}. We now proceed with some of the definitions that were informally given in the introduction:

\begin{definition}[sufficiently dense]\label{d:dense}
We say that a graph sequence $H=(H_n)_{n\ge1}$ is \emph{sufficently dense}, or simply \emph{dense}, if
	\[
	\lim_{n\to\infty}v(H_n)
	=\infty
	=\lim_{n\to\infty} \frac{e(H_n)}{v(H_n)\log v(H_n)}\,,
	\]
where $v(H)$ denotes the number of vertices in $H$,
and $|H|\equiv e(H)$ denotes the number of edges in $H$.
\end{definition}
% In the context of planted subgraph model, the quantity $M$ in Definition~\ref{d:bern.inference} correspond to the total number of distinct copies of $H$ in $K_n$. Thus, we define
% 	\beq\label{e:M.H}
% 	M \equiv M_H \equiv M_{H,K_n}
% 	= \frac{n! / (n-v(H))!}{|\Aut(H)|} = \frac{(n)_{v(H)}}{|\Aut(H)|}
% 	= \binom{n}{v(H)} \Psi_H
% 	\,,
% 	\eeq
% where $\Aut(H)$ denotes the automorphism group of $H$,
% and $\Psi_H$ denotes the number of distinct copies of $H$ inside the complete graph on $v(H)$ vertices. Recalling the first moment threshold for the Bernoulli model (cf. equation \ref{def:p1M}), the \emph{first moment threshold} of $H$ is defined by

%We now define a generalization of the \textit{expectation threshold} (cf. \ys{cite Kahn Kalai}), which turns out to be a crucial quantity for all-or-nothing.

For dense graphs, $\pmmt(H)$ and $\PSI_q(H)$ are easily computable by the following lemma, the proof of which was already sketched in the introduction (see \eqref{e:dense.threshold}):

\begin{lemma}\label{l:dense.threshold}
Let $H_n$ be dense (cf. Definition~\ref{d:dense}). Then, we have
	\[
	\pmmt(H_n) = \frac{1+o_n(1)}{n^{v(H)/e(H)}}\,.
	\]
Thus, if we let
	\beq\label{e:alpha.q}
	\alpha_q 
	\equiv \alpha_q(H)
	\equiv \inf\bigg\{
	\frac{v(J)}{e(J)}
	: J\subseteq H
	\textup{ with }
	|J| \ge\max\Big\{1, |H| q\Big\}
	\bigg\}\,,
	\eeq
then for any dense graph $H_n$, we have $\PSI_q (H)=(1+o_n(1))n^{-\alpha_q(H)}$ for $0<q\leq 1$. Moreover, for an arbitrary graph $H$, we have the lower bound $\PSI_q(H)\geq n^{-\alpha_q(H)}$ for $q\in [0,1]$.
\end{lemma}

\begin{proof}
It follows from the definitions that
	\[
	\frac{1}{\pmmt(H)}
	= (M_H)^{1/e(H)}
	=\bigg( \frac{(n)_{v(H)}}{|\Aut H|}\bigg)^{1/e(H)}\,.
	\]
We then note that the dense assumption implies
	\[
	1 \le |\Aut H|^{1/e(H)}
	\le (v(H)!)^{1/e(H)}
	= \exp \bigg\{O\bigg( \frac{v(H)\log v(H)}{e(H)}\bigg)\bigg\}
	= e^{o_n(1)}\,.
	\]
Similarly, it also implies
	\[
	\frac{( (n)_{v(H)})^{1/e(H)}}{  n^{v(H)/e(H)}}
	=
	\exp\bigg\{ O \bigg( \frac{v(H)^2}{n e(H)} \bigg)\bigg\}
	= e^{o_n(1)}\,.
	\]
For the final claim regarding the one-sided bound of $\psi_q(H)$, note that for any graph $J$,
\begin{equation*}
    \pmmt(J)=\bigg( \frac{(n)_{v(J)}}{|\Aut J|}\bigg)^{-1/e(J)}\geq \Big( (n)_{v(J)}\Big)^{-1/e(J)}\geq n^{-v(J)/e(J)}.
\end{equation*}
The claim follows.
\end{proof}

The following condition was introduced in \eqref{e:intro.delocalized} and used in the statement of Theorem~\ref{t:dense.linear.aon}:

\begin{definition}[delocalized]\label{d:delocalized}
Let $H=(H_n)_{n\ge1}$ be a sequence of graphs that are dense. We say the sequence is \emph{delocalized} if
there exist $0<q\leq 1 $ and $C<\infty$, which are independent of $n$, such that
	\beq\label{e:delocalized}
	\alpha_q(H)\leq \alpha_0(H)+\frac{C}{\log n}.
	\eeq
We remark that if $H=(H_n)_{n\geq 1}$ is dense and satisfies $\PSI_q(H)\geq c\cdot \PSI_0(H)$ for some $0<q\leq 1$ and $c>0$, then $H$ is delocalized: by Lemma~\ref{l:dense.threshold}, we have
    \[
    \frac{1+o_n(1)}{n^{\alpha_q(H)}}
    = \PSI_q(H)
    \ge c \cdot \PSI_0(H)
    \ge \frac{1}{n^{\alpha_0(H)}}\,,
    \]
and rearranging gives condition \eqref{e:delocalized}. This fact was used in some of the examples presented in Section~\ref{ss:main.results}.
\end{definition}

\subsection{Theorem~\ref{t:dense.linear.aon} forward direction: AoN implies almost-balanced} 
\label{ss:aon.lin.forward}

We now turn to the proof of Theorem~\ref{t:dense.linear.aon}. 

\begin{lemma}\label{lem:count:bound}
Let $\bP_1$ be the uniform measure on copies of $H$ in $K_n$, and let $\bP_2$ be the uniform measure on copies of $H'$ in $K_n$. Suppose $v(H) \le v(H')$. If $(\bH^1,\bH^2)$ denotes a sample from $\bP_1\otimes\bP_2$, then
	\[
	(\bP_1\otimes\bP_2)(|\bH^1\cap\bH^2|=\ell)
	\le \frac{2^{v(H')}
	v(H)^{O(v(H))}}{ n^{v_\ell(H,H')}}
	\,,
	\]
where $v_\ell(H,H')$ denotes the minimum number of vertices in a graph of $\ell$ edges that can arise as an intersection of a copy of $H$ with a copy of $H'$.
\end{lemma}

\begin{proof}
It follows by a direct counting argument that
	\[
	(\bP_1\otimes\bP_2)(|\bH^1\cap\bH^2|=\ell)
	\le \frac1{M_H}
	\sum_{w=v_\ell(H,H')}^{v(H)}
	\binom{v(H')}{w}
	\binom{n-v(H')}{v(H)-w} \Psi_H\,,
	\]
 We can simplify the above bound as
	\begin{align*}
	&(\bP_1\otimes\bP_2)(|\bH^1\cap\bH^2|=\ell)
	\le
	\binom{n}{v(H)}^{-1}
	\sum_{w=v_\ell(H,H')}^{v(H)}
	\binom{v(H')}{w}
	\frac{(n-v(H'))_{v(H)-w}}{(v(H)-w)!}\\
	&\qquad\le 2^{{\BG v(H')}}
	\sum_{w=v_\ell(H,H')}^{v(H)}
	\frac{(v(H))!}{(v(H)-w)!}
	\frac{n^{v(H)-w}}{ (n)_{v(H)} }
	\le 
	\frac{2^{v(H')}
	v(H)^{O(v(H))}}{ n^{v_\ell(H,H')}}\,,
	\end{align*}
as claimed.\end{proof}

The next two lemmas give sufficient conditions for bounding the normalized MMSE from below (Lemma \ref{l:not.all}), and from above (Lemma \ref{l:not.nothing}). % \BM The next lemma is similar to Lemma~\ref{l:not.all.exp}. \EH

\begin{lemma}\label{l:not.all}
Suppose $H=H_n$ is dense in the sense of Definition~\ref{d:dense}.
Fix $q\in(0,1]$. We then have 
	\[\frac{\mmse_n(p_n)}{K}
	\ge 1- \BH q - o_n(1).\]
as long as $p \ge \BH (1+\epsilon)/n^{\alpha_q(H)}$ 
for $\alpha_q(H)$ as defined by \eqref{e:alpha.q}.
\end{lemma}

\begin{proof}
By Lemma~\ref{l:mmse.lbd}, it suffices to check condition~\eqref{e:mmse.lbd.condition} with $\delta=q$. By Lemma~\ref{lem:count:bound},
	\[
	\bP^{\otimes 2}(|\bH \cap \bH'| =\ell)
	\le
	\frac{ v(H)^{O(v(H))} }{n^{v_\ell(H)} }\]
where $v_\ell(H)\equiv v_\ell(H,H)$.
It follows from the definition~\eqref{e:alpha.q} that 
	\[
	\frac{v_\ell(H)}{\ell} \ge \alpha_q(H)\,.
	\]
Consequently, as long as $p \ge (1+\epsilon)/n^{\alpha_q(H)}$, we have
	\[\frac{\bP^{\otimes 2}(|\bH \cap \bH'| =\ell)}{p^\ell}
	\le
	\frac{ v(H)^{O(v(H))} }{ ( 1+\epsilon )^\ell }\,.
	\]
Summing over $\ell\ge Kq$ and recalling the dense assumption gives
\eqref{e:mmse.lbd.condition}.
\end{proof}

\begin{lemma}\label{l:not.nothing}
Suppose $H=H_n$ is dense in the sense of Definition~\ref{d:dense}.
With the notation \eqref{e:alpha.q}, suppose that
	\[
	\alpha_q(H) \le \alpha_0(H) + \frac{c}{\log n}
	\]
for some $q\in(0,1)$, where $c$ is a finite constant. Then for small $\eta>0$ we have 
	\[\frac{\mmse_n(p_{n_k})}{K}
	\le 1- \BH \eta - o_n(1)\]
as long as $p\le (1- 2 c \eta /q \EH ) n^{-\alpha_q(H)}$.
\end{lemma}

\begin{proof}
Suppose for some $q\in(0,1)$ we have
	\[
	\alpha_q(H) = \alpha_0(H) + \frac{\BH c_q}{\log n}
	\]
for some $q\in(0,1)$, where $\BH c_q$ may depend on $n$ but stays bounded by $c$ in the limit $n\to\infty$.
Let $J\subseteq H$ be a subgraph with $L\equiv e(J)\ge Kq$
and $v(J)/e(J)=\alpha_q(H)$. Let $\mathcal{S}$ denote the set of all copies of $H$ in $K_n$. As in the proof of Lemma~\ref{l:mmse.ubd}, it suffices to show that, with high probability, any element $H'\in\mathcal{S}$ satisfying $H'\subseteq \bY=\bH\cup\bG$ has overlap at least $K\eta$ with the planted subgraph $\bH$. To this end, let $\mathcal{J}$ denote the set of all copies of $J$ in $K_n$; it then suffices to show that any element $J'\in\mathcal{J}$ satisfying $J'\subseteq\bY$ has overlap at least $K\eta$ with $\bH$. Similarly as in Lemma~\ref{l:mmse.ubd}, let 
$\bar{\bZ}_\ell(\bS,\bY)$ denote the number of elements $J'\in\mathcal{J}$ that are contained in $\bY$ and have overlap $\ell$ with $\bH$; we want to show that with high probability this quantity is zero for all $\ell\le K\eta$. Similarly to \eqref{e:all.bound.A}, we have
    \[
    q_\ell\equiv 
    \E_{\P} \bar{\bZ}_\ell(\bS,\bY)
    = (\bP\otimes \bar{\bP}) (|\bH\cap\bJ | =\ell )
    M_J p^{L-\ell}\,,
    \]
where $\bar{\bP}$ is the uniform measure over $\mathcal{J}$, and
$M_J\equiv |\mathcal{J}|$. Applying Lemma~\ref{lem:count:bound} gives \EH 
    \[
    q_\ell
    \le 
    \frac{(v_H)^{O(v_H)}}{n^{v_\ell(J,H)}}
    M_J p^{L-\ell}
    \le 
    (v_H)^{O(v_H)}
    \frac{n^{v_J-v_\ell(J,H)}}{(1/p)^{L-\ell}}\,.
    \]
It follows from the definition \eqref{e:alpha.q} that
$v_\ell(J,H)/\ell \ge \alpha_0(H)=\alpha_0$,
and $v(J)/L =\alpha_q(H)=\alpha_q$. Combining with the assumption on $\alpha$ gives \EH 
	\[q_\ell
	\le (v_H)^{O(v_H)} 
		\frac{n^{L\alpha_q-\ell\alpha_0}}{(1/p)^{L-\ell}}
	\le (v_H)^{O(v_H)}
	e^{Lc_q}
	\bigg(\frac{n^{\alpha_0} }{1/p}\bigg)^{L-\ell} \,.
	\]
Let $\epsilon>0$ be an arbitrary constant for now, and set
	\[
	p = \frac{1-\epsilon}{n^{\alpha_0} \exp(c_q/(1-\eta 
	    /q \EH ))}
	\ge \frac{(1-\epsilon - (3/2)c_q\eta /q \EH )}
	    {n^{\alpha_q}}\,,
	\]
where the last estimate holds for $\eta$ small enough. For $\ell\le K\eta \le L\eta/q$, we have
    \[
    q_\ell
    \le\frac{ (v_H)^{O(v_H)}
    (1-\epsilon)^{L-\ell} e^{Lc_q}}
    {\exp( (L-\ell) c_q/(1-\eta /q))}
    \le \frac{\exp(o(K))}{\exp( L\epsilon(1-\eta/q))}\,,
    \]
where the last bound uses the dense assumption. Summing over $\ell \le K\eta$ proves
    \[
    \E_{\P} \sum_{\ell\le K\eta}
    \bar{\bZ}_\ell(\bS,\bY)
    =\sum_{\ell\le K\eta} q_\ell 
    = o_n(1)\,,
    \]
which implies the desired bound on MMSE. The claim follows by taking $\epsilon=c_q\eta/(2q)$.
\end{proof}

\begin{proof}[Proof of Theorem~\ref{t:dense.linear.aon} forward direction]
Let $\alpha_q\equiv\alpha_q(H)$ as in \eqref{e:alpha.q}.
Lemma~\ref{l:not.all} tells us that
\BH ``all'' cannot hold for \EH
	\[
	p \BH \ge \EH \frac{1+o_n(1)}{n^{\alpha_q}}
	\]
for any $q\in(0,1)$; since we assume AoN, it means that ``nothing'' holds. On the other hand, using the delocalized assumption (Definition~\ref{d:delocalized}), Lemma~\ref{l:not.nothing} tells us that \BH ``nothing'' cannot hold for \EH
	\[
	p \BH \le \EH \frac{1-o_n(1)}{n^{\alpha_{q'}}}
	\]
for sufficiently small $q'>0$; again, since we assume AoN, it means that ``all'' holds. If $q' \le q$ then $\alpha_{q'} \le \alpha_q$. We therefore obtain a contradiction unless
	\[
	\alpha_q - \alpha_{q'}
	= o\bigg(\frac{1}{\log n}\bigg)
	\]
for all $q,q'\in(0,1)$, and $\pAoN=(1+o_n(1)) /n^{\alpha_q}$ for all $q\in(0,1)$.
\end{proof}

\subsection{Theorem~\ref{t:dense.linear.aon} reverse direction: almost-balanced implies AoN} 
\label{ss:aon.lin.reverse}

The following is a strengthening of Lemma~\ref{l:not.nothing} under the assumption~\eqref{e:almost.bal}.

\begin{lemma}\label{lem:dense:all}
Suppose $H=(H_n)_{n\ge1}$ is delocalized in the sense of Definition~\ref{d:delocalized}. If condition~\eqref{e:almost.bal} holds, then
	\[\lim_{n\to\infty} \mmse_n(p) = 0\]
for all $p\le (1-\epsilon)/n^{\alpha_q}$, for any $q\in(0,1)$.
\end{lemma}

\begin{proof}
Fix small enough $\delta>0$. Write $K=|H|$, and let $J$ be the subgraph of $H$ which achieves maximum density among all the subgraphs of $H$ having at least $K(1-\delta/2)$ edges, and write $L=|J|$. Let $\mathcal{J}$ denote the set of all copies of $J$ in $K_n$.
Similarly to the proof of Lemma~\ref{l:not.nothing}, it suffices to show that, with high probability, any element $J'\in \mathcal{S}$ satisfying $J'\subseteq\bY=\bH\cup\bG$ has overlap at least $K(1-\delta)$ with the planted subgraph $\bH$. For this it suffices to show
    \beq\label{e:mmse.ubd.condition.repeated}
    \E_{\P} \sum_{\ell\le K(1-\delta)}
    \bar{\bZ}_\ell(\bS,\bY)
    =\sum_{\ell\le K(1-\delta)}
        q_\ell 
    = o_n(1)\,,
    \eeq
where, as in the proof of
Lemma~\ref{l:not.nothing}, we have 
    \[
    q_\ell
    \le (v_H)^{O(v_H)} 
    \frac{n^{v(J)-v_\ell(J,H)}}{(1/p)^{L-\ell}}\,.
    \]
We then divide the bound into two regimes:
\begin{enumerate}[(a)]
\item For $\ell\le L\eta$, we have $v(J)/L = \alpha\equiv \alpha_{1-\delta/2}(H)$ by assumption, while
	\[
	\frac{v_\ell(J,H)}{\ell}
	\ge \alpha_0(H) 
    \equiv \alpha_0 
    \ge {\alpha_{q'}} - \frac{c}{\log n}
    \stackrel{\eqref{e:almost.bal}}{=}
    \alpha_q - \frac{c + o_n(1)}{\log n}
	\]
where the intermediate bound follows from the delocalized assumption (Definition~\ref{d:delocalized}) for small enough $q'$, and the last equality holds for any $q'$ by \eqref{e:almost.bal}. Thus, for $\ell\le L\eta$ and $p \le (1-\epsilon)/n^{\alpha_q}$, we have
	\[
	q_\ell
	\le 
	(v_H)^{O(v_H)} e^{cL\eta}
	\bigg(\frac{n^\alpha}{1/p}\bigg)^{L-\ell}
	\le   \frac{
   (v_H)^{O(v_H)} 
    e^{cL\eta}}
	{\exp(L[(1-\eta)\epsilon + o_n(1)])}
	\]
It follows by taking $\eta=\epsilon/(2c)$ and summing over $\ell\le L\eta$
(using the dense assumption) that the contribution to \eqref{e:mmse.ubd.condition.repeated} from all such $\ell$ is $o_n(1)$.

\item For $L\eta = L\epsilon/(2c) \le\ell\le L(1-\delta/2)$, we have
	\[
	\frac{v_\ell(J)}{\ell} \ge \alpha_\eta(J)\equiv \alpha_\eta
	\stackrel{\eqref{e:almost.bal}}{=} 
 {\BRR\alpha_q}
 +o\bigg(\frac{1}{\log n}\bigg)\,.
	\]
It follows that
	\[q_\ell
	\le (v_H)^{O(v_H)}
	\frac{n^{L\alpha
    -\ell\alpha_\eta}}{(1/p)^{L-\ell}}
	\le e^{o(K)}
	\bigg(\frac{n^\alpha}
    {1/p}\bigg)^{L-\ell}
	\le \frac{\exp(o(K))}{\exp(L\delta\epsilon/2)}\,,
	\]
having used the dense assumption.
It follows by summing over $L\eta\le \ell \le L(1-\delta/2)$
that the contribution to \eqref{e:mmse.ubd.condition.repeated} 
from all such $\ell$ is $o_n(1)$.
\end{enumerate}
Combining the above bounds proves \eqref{e:mmse.ubd.condition.repeated}, and hence the claim.
\end{proof}

\begin{proof}[Proof of Theorem~\ref{t:dense.linear.aon} reverse direction]
Let $\alpha=\alpha_q$ for any fixed $q\in(0,1)$, and let $\epsilon>0$ be any positive constant.
Lemma~\ref{lem:dense:all} tells us that ``all'' occurs for $p \le (1-\epsilon)n^{-\alpha}$.
On the other hand, Lemma~\ref{l:not.all} tells us that ``nothing'' occurs for all $p \ge (1+\epsilon)n^{-\alpha}$. This proves the theorem.
\end{proof}

\subsection{Sparse strongly balanced graphs}
\label{ss:sparse}

Given the result of Theorem~\ref{t:dense.linear.aon}, it is natural to ask about AoN for general graphs without a density requirement as in Definition~\ref{d:dense}. In this subsection we show that AoN at linear holds for graphs that are \emph{strongly balanced} (see below) and have a sufficiently small number of vertices and edges, but with no restriction on the density. We do not have a complete understanding of the sparse case, and leave this as an interesting open question for future investigations.

\begin{definition}[strongly balanced]\label{d:strongly.balanced}
We say the graph $H$ is \emph{strongly balanced} with parameter $c>0$ if we have
	\beq\label{eq:strict:balanced}
        \frac{|J|}{v(J)-c}
        \le\frac{|H|}{v(H)-c}
	%\frac{v(J)}{|J|} \ge \frac{v(H_n)}{|H_n|}
	%+\frac{c\delta}{|J|},
	\eeq
for all nonempty subgraphs $J\subseteq H$. That is to say, subgraphs of $H_n$ must be \emph{strictly} less dense than $H_n$, with the difference given by \eqref{eq:strict:balanced}. The case $c=1$ corresponds to the definition of \cite{MR890231}.
\end{definition}

\begin{theorem}[AoN at linear scale for small strongly balanced graphs]\label{t:aon.sparse}
Suppose $H=H_n$ is strongly balanced (Definition~\ref{d:strongly.balanced}) with parameter $c>0$, and satisfies $v(H_n)\to\infty$, $|H_n|\to\infty$, and 
	\[
	v(H_n)+|H_n| 
	\le \frac{c\log n}{3\log\log n}\,.
	\]
Then the model of $H_n$ planted in $G(n,p)$ exhibits AoN at 
$\pAoN=\pmmt(H)$.
\end{theorem}

\begin{example}[small trees and small cycles]\label{x:small.tree.cycle}
One can easily check that trees and cycles are strongly balanced with parameter $c=1$; this fact is also noted by \cite{MR890231}. Theorem~\ref{t:aon.sparse} implies that if $H_n$ is a tree or cycle with
	\beq\label{e:tree.cycle.bound}
	v(H_n)
	\le \frac{\log n}{6\log\log n}\,,
	\eeq
then the model of $H_n$ planted in $G(n,p)$ exhibits AoN at linear scale at $\pAoN=\pmmt(H)$. Compare with Example~\ref{x:cycle}, which addresses AoN at the exponential scale for cycles much larger than \eqref{e:tree.cycle.bound}. We also note that Theorem~\ref{t:aon.sparse} does not apply to the $k$-cycle with $k$ extra edges  (Example~\ref{x:cycle.out}), since this is not strongly balanced. \EH
\end{example}

In preparation for the proof of Theorem~\ref{t:aon.sparse} we introduce some new notation.  For a subgraph $J\subseteq H$, in keeping with our earlier notation 
\eqref{e:M.H} let $M_{J,H}$ denote the number of copies of $J$ in $H$.
Given a copy $J_0$ of $J$ in $K_n$, let $M_{H|J}$ denote the number of ways to extend this to a copy of $H$: we can rewrite this as
	\[
	\frac{M_{H|J}}{M_H}
	=\bP(J_0\subseteq \bH)
	\]
where $\bH\sim\bP$. Note also that if $\bJ\sim \bP_J$ denotes a uniform copy of $J$ in $K_n$, then we have
	\beq\label{e:double.count}
	\frac{M_{H|J}}{M_H}
	=\bP(J_0\subseteq \bH)
	=\bP_J (\bJ \subseteq H_0)
	= \frac{M_{J,H}}{M_J}\,,
	\eeq
where $H_0$ denotes a fixed copy of $H$ in $K_n$.

\begin{proof}[Proof of Theorem~\ref{t:aon.sparse}]
We start by remarking that the conditions on $v(H_n)$ and $e(H_n)$ imply
	\beq\label{e:rearranged.condition}
	2\Big(v(H) + |H|\Big)\log v(H)
	-c\log n
	\le \frac{2c\log n}{3[1-o_n(1)]} - c\log n
	\le -\frac{c\log n}{4}\,.
	\eeq
\BM By Theorem~\ref{t:bern.aon.linear}, \EH it suffices to check the condition~\eqref{e:growth}. For $q\in[0,1]$, we now bound
	\[
	F_q
	\equiv
	\frac{\bP^{\otimes2}
	(|\bH\cap\bH'|=|H|q)}{\pmmt(H)^{|H|q}}
	\le
	\sum_{\substack{F\subseteq H,\\|F|\ge |H|q}}
	\mathbf{1}\Big\{\textup{$F$ vertex-induced}\Big\}
	\sum_{\substack{J\subseteq F,\\|J|= |H|q}}
	\mathbf{1}\Big\{\textup{$V(J)=V(F)$}\Big\}
	\frac{\BH M_{H|J}}{(M_H)^{1-q}}
	\,.
	\]
To form another copy $H'\subseteq K_n$ of $H$ with $H\cap H'=J$, we need to choose $v(H)-v(F)$ vertices from the $n-v(H)$ available vertices, and also choose $|H|(1-q)$ edges among the at most $v(H)^2$ possible edges. This gives a crude bound
	\[
	M_{H|J}
	\le \binom{n-v(H)}{v(H)-v(F)}
	(v(H)^2)^{|H|(1-q)}\,.
	\]
Next note that the strongly balanced assumption implies
	\[
	\frac{v(J)}{|J|}
	\ge \frac{v(H)}{H} + \frac{c(1-q)}{|J|}\,,
	\]
so $v(J) \ge v_{|H|q}(H) \ge v(H) q+ c(1-q)$. 
The strongly balanced assumption also implies
    \[
	|F| \le \frac{|H|(v(F)-c)}{v(H)-c}\,.
	\]
If we then account for the enumeration of $F$ and $J$, we obtain
	\[
	F_q
	\le\frac{(v(H)^2)^{|H|(1-q)}}{(M_H)^{1-q}} 
	\sum_{w=v(H)q + c(1-q)}^{v(H)}
	\binom{v(H)}{w}
	\bigg( \frac{|H|(w-c)}{v(H)-c}\bigg)^{|H| \frac{w-c}{v(H)-c}-|H|q}
	\binom{n-v(H)}{v(H)-w}\,,\]
where the first binomial coefficient accounts for the number of vertex-induced subgraphs $F$. Simplifying the bound gives
	\begin{align*}
	F_q&\le
	\frac{v(H)^{2|H|(1-q)}}{(M_H)^{1-q}} 
	\sum_{w=v(H)q + c(1-q)}^{v(H)}
	\Big[v(H) (n-v(H)\Big]^{v(H)-w}
	\bigg( \frac{|H|(w-c)}{v(H)-c}\bigg)^{|H| \frac{w-c}{v(H)-c}-|H|q}\\
	&\equiv
	\frac{v(H)^{2|H|(1-q)}}{(M_H)^{1-q}} 
	\sum_{w=v(H)q + c(1-q)}^{v(H)}
	\exp f_q(w)\,.
	\end{align*}
We claim that $f_q$ is decreasing in $w$, so that each term in the last sum is upper bounded by
	\[
	f_q\Big(v(H)q + c(1-q)\Big)
	=
	(v(H)-c)(1-q)
	\log\bigg\{v(H) (n-v(H))\bigg\}\,.
	\]
We also have the trivial inequality
	\[
	\frac{1}{(M_H)^{1-q}}
	=1\bigg/ \bigg( \binom{n}{v(H)} \Psi_H\bigg)^{1-q}
	\le1\bigg/ \binom{n}{v(H)}^{1-q}
	\le \frac{1}{(n-v(H))^{v(H)(1-q)}}
	\]
Combining the above bounds gives
	\[
	F_q
	\le
	v(H) \cdot v(H)^{2|H|(1-q)}
	\frac{[v(H) (n-v(H))]^{(v(H)-c)(1-q)}}{(n-v(H))^{v(H)(1-q)}}
	\le \bigg( \frac{v(H)^{2(|H|+v(H))}}{(n-v(H))^c}\bigg)^{1-q}
	\,.
	\]
so it follows from~\eqref{e:rearranged.condition} that $F_q \le 1 \le \exp(o(H))$, which gives condition~\eqref{e:growth}.
It remains to verify that $f_q$ is in fact nonincreasing in $w$ as claimed above. To this end we bound
	\begin{align*}
	\frac{d}{dw} f_q(w)
	&= - \log\Big[v(H) (n-v(H)\Big]
	+\frac{|H|}{v(H)-c}
	\log \bigg( \frac{|H|(w-c)}{v(H)-c}\bigg)
	+ \frac{|H|}{v(H)-c}-\frac{|H|q}{w-c}\\
	&\le
	-\log\Big[2(n-2)\Big]
	+ \frac{|H|\log |H|}{v(H)-c}
	+\frac{|H|(1-q)}{v(H)-c}\,,
	\end{align*}
having used that $v(H)\ge2$ (which clearly holds for large $n$, since we assumed $v(H_n)\to\infty$). Then, since $v(H_n)\to\infty$,
$|H_n|\to\infty$, and
$|H| \le v(H)^2$, we can further bound
	\[
	\frac{d}{dw} f_q(w)
	\le
	-\log n
	+ \frac{3|H|\log v(H)}{v(H)}
	\le
	-\log n
	+ \frac{|H|\log v(H)}{c}
	\le0\,,
	\]
where the last inequality again holds by \eqref{e:rearranged.condition}. This concludes the proof.
\end{proof}

\section{AoN at exponential scale}
\label{s:exp}

Recall the generalized expectation threshold $\PSI_q(H)$ from Definition~\ref{d:gen.expectation.threshold}, and let us define
    \beq\label{e:lambda.q}
    \lambda_q(H)
    \equiv \log \frac{1}{\PSI_q(H)}\,.
    \eeq
We then introduce the following:

% for graphs of potentially growing size. The definition is a property of the first moment thresholds of the various subgraphs of $\BH H$. It can be easily checked to be equivalent with the known notion of balancedness for constant sized graphs.
\begin{definition}[first-moment-stable]
\label{d:first.mmt.stable}
We say the graph sequence $H=(H_n)_{n\ge1}$ is \emph{first-moment-stable} if it holds uniformly in $n$ that
	\[
	\lim_{q\uparrow 1}\lambda_q=\lambda_1\,,
	\]
for $\lambda_q$ as in \eqref{e:lambda.q}. Informally this says that if $|J|=|H|(1-o(1))$ then we must have
$\pmmt(J) = \pmmt(H)^{1+o(1)}$. We will see that
it is equivalent to require $\pmmt(J) \le \pmmt(H)^{1-o(1)}$;
the other direction holds automatically.
\end{definition}

\begin{definition}[first-moment-flat]
\label{d:first.mmt.flat}
We say that a graph sequence $H=(H_n)_{n\ge1}$ is \emph{first-moment-flat} if
	\[
	\frac{\lambda_q(H_n)}{\lambda_{q'}(H_n)}
	= 1+o_n(1)
	\]
for all $q,q'\in(0,1)$.
\end{definition} 

The conditions from
Definitions~\ref{d:first.mmt.stable} and \ref{d:first.mmt.flat} should be compared with 
the conditions
that appeared in the result for dense graphs Theorem~\ref{t:dense.linear.aon}:
delocalized
(Definition~\ref{d:delocalized}) and
almost-balanced
(condition~\eqref{e:almost.bal}). \BH To give a concrete instance, the clique with out-edges discussed in Examples~\ref{example:sun} and \ref{example:sun:generalize} is delocalized, but not first-moment-stable.
\EH 

\begin{theorem}[AoN at exponential scale]\label{t:aon.exp}
Let $H=(H_n)_{n\ge1}$ be a graph sequence with 
$|H_n|\to\infty$ and $\pmmt(H_n)$ bounded away from one.
Suppose further that $H$ is first-moment-stable in the sense of Definition~\ref{d:first.mmt.stable}, and also that either $\pmmt(H_n)\to0$ or $v(H)\le n^{o(1)}$. Then AoN occurs at the exponential scale if and only if $H$ is first-moment-flat (Definition~\ref{d:first.mmt.flat}). In this case $\pAoN=\pmmt(H)$.
\end{theorem}

The main purpose of this section is to prove Theorem~\ref{t:aon.exp}, which characterizes AoN at exponential scale for graphs that either have first moment threshold tending to zero, or
are subpolynomial in size. This section is organized as follows:

\begin{itemize}
\item In \S\ref{ss:rws} we return to the generalized setting of Definition~\ref{d:bern.inference}, and show that if the prior $\bP$ satisfies a certain ``replica weak separation'' (RWS) condition, then AoN at exponential scale can \emph{only} occur at the first moment threshold.

\item In \S\ref{ss:first.mmt.stable} we show that if the graph
sequence $H=(H_n)_{n\ge1}$ is first-moment-stable (Definition~\ref{d:first.mmt.stable}), then the uniform measure $\bP$ on copies of $H$ in $K_n$ satisfies RWS.

\item In \S\ref{ss:exp.forward} we prove the forward direction of Theorem~\ref{t:aon.exp} by showing if that the planted subgraph model has AoN at the exponential scale, then the graph sequence must be first-moment-flat.

\item \S\ref{ss:exp.reverse} we prove the reverse direction of Theorem~\ref{t:aon.exp} by showing if that the graph sequence
is first-moment flat, then the corresponding planted subgraph model has AoN at the exponential scale.
\end{itemize}
Theorem~\ref{t:aon.exp} \BM goes beyond the dense regime --- this is illustrated by
Examples~\ref{x:matching}, \ref{x:small.bal}, \ref{x:cycle.out} from \S\ref{ss:beyond}, as well as by the following: \EH

\begin{example}[cycle]\label{x:cycle}
%We next present two examples involving cycles (of diverging size), to illustrate our results for sparse graphs.
Let $H$ be a cycle on $k$ vertices. If $J\subseteq H$ with $1<|J|=\ell<k$, then $J$ must consist of $p$ disjoint paths, with $1\le p \le\ell$,
of lengths $\ell_1,\ldots,\ell_p$ summing to $\ell$. It follows that $v(J)=\ell+p$. We also note that an automorphism of $J$ is determined by how it acts on the endpoints of the paths, so crudely
$|\Aut J| \le (2p)^p$.
It follows that
	\[
	\pmmt(J)
    \stackrel{\eqref{e:graph.first.mmt.threshold}}{=}
	\bigg( \frac{1}{M_J}\bigg)^{1/|J|}
	=\Big(1+o(1)\Big)
	\frac{|\Aut J|^{1/|J|}}{ n^{v(J) / |J|}}
	\le
	\frac{1+o(1)}{n}
	\bigg(\frac{2p}{n}\bigg)^{p/\ell}
	\le \frac{1+o(1)}{\pmmt(H)}\,.
	\]
This shows that $J$ is first-moment-stable (Definition~\ref{d:first.mmt.stable})
and first-moment-flat (Definition~\ref{d:first.mmt.flat}).
It then follows from Theorem~\ref{t:aon.exp}
that the model of $H_n$ planted in $G(n,p)$ exhibits AoN at exponential scale at $\pmmt(H) = (1+o(1))/n$. Compare with Example~\ref{x:small.tree.cycle} (which is restricted to smaller cycles, but gives AoN at linear scale), and with Example~\ref{x:cycle.out} (the cycle with out-edges).
\end{example}

\BM We also include an example where AoN does not occur even at the exponential scale: \EH

\begin{example}[lack of AoN at exponential scale] \label{x:two.cliques}
Take $1 \ll k_2 \le k_1 \le O(\log n)$ where the ratio $k_2/k_1$ is a \BH small constant $\delta$. \EH Let $H=H_n$ be the disjoint union of $J_1,J_2$ where each $J_i$ is a clique on $k_i$ vertices. Then
	\[
	1-\epsilon \le \frac{|J_1|}{|H|} \le 1-\frac{\epsilon}{2}
	\]
for a small constant $\BH\epsilon = (1+o(1)) \delta^2$. It follows that
	\[
	\lambda_{1-\epsilon}(H)
	\le \log \frac{1}{\pmmt(J_1)}
	= \frac{2}{k_1-1}\log n + o(1)
	\]
On the other hand, for the full graph we have
	\[
	\lambda_1(H)
	= \log\frac{1}{\pmmt(H)}
	= \frac{2(1+\Theta(\delta))}{k_1-1} \log n+ o(1)\,,
	\]
so $H$ is not first-moment-flat (Definition~\ref{d:first.mmt.flat}). \BH It is however first-moment-stable (Definition~\ref{d:first.mmt.stable}). \EH It follows from Theorem~\ref{t:aon.exp} that the model of $H_n$ planted in $G(n,p)$ does not have AoN at the exponential scale. Indeed, 
for $\pmmt(J_2) < p < \pmmt(J_1)$ on the exponential scale, it will be possible to recover $J_1$ but not $J_2$.
\end{example}

\subsection{Locating the AoN threshold under RWS condition} 
\label{ss:rws}
In this subsection we show that if $\bP$ satisfies a ``replica weak separation'' condition (Definition~\ref{d:rws} below), then at the exponential scale it can \emph{only} exhibit all-or-nothing at the first moment threshold $\pmmt$ from \eqref{def:p1M}. 
%The condition is weaker than the previous assumptions \eqref{e:growth} and \eqref{e:weak.growth}:  

\begin{definition}[replica weak separation]\label{d:rws}
We say the measure $\bP$ satisfies \emph{replica weak separation} (RWS) if
\beq\label{e:rws}
\adjustlimits
\lim_{\delta\downarrow0} \limsup_{N \to \infty} 
\frac{1}{\log M}
	\log\bigg\{
	\bP^{\otimes 2}\Big(|\bS \cap \bS'| \ge (1-\delta)K \Big) 
	M
	\bigg\}=0\,,\eeq
where $\bS,\bS'$ are independent samples from $\bP$.
\end{definition}

In \S\ref{ss:first.mmt.stable} we show that if $H$ is first-moment-stable in the sense of Definition~\ref{d:first.mmt.stable}, then the uniform measure on copies of $H$ in $K_n$ satisfies the RWS condition of Definition~\ref{d:rws}. The following lemma explains how this is related to the earlier condition \eqref{e:weak.growth} which was used in the proof of Theorem~\ref{t:bern.aon.exp}:

\begin{lemma} 
If the prior satisfies \eqref{e:weak.growth} and $\log K \ll \log M$,
then it also satisfies \eqref{e:rws}.
\end{lemma}

\begin{proof}
Recall from \eqref{def:p1M} that $\pmmt=M^{-1/K}$. It follows that
	\begin{align*}
	&\bP^{\otimes 2}\Big(|\bS \cap \bS'| \ge (1-\delta)K \Big) 
	M
	= \sum_{\ell = K(1-\delta)}^K
	\frac{\bP^{\otimes 2}(|\bS \cap \bS'| =\ell )}{(\pmmt)^\ell}
	M^{1-\ell/K} \\
	&\qquad
	\le
	M^{\delta}  \sum_{\ell = K(1-\delta)}^K
		\frac{\bP^{\otimes 2}(|\bS \cap \bS'| =\ell )}{(\pmmt)^\ell}
	\stackrel{\eqref{e:weak.growth}}{\le} 
	M^\delta  K\delta
	\exp(o(\log M))\,.
	\end{align*}
The claim follows by recalling the assumption $\log K\ll \log M$ and sending $\delta\downarrow0$.
\end{proof}

\begin{theorem}[location of AoN threshold at exponential scale]
\label{thm:locating}
Assume $\bP$ satisfies RWS in the sense of condition~\eqref{e:rws} from Definition~\ref{d:rws}. If the corresponding planted model has all-or-nothing at the exponential scale at some critical value $\pAoN$, then
\[ \lim_{N\to\infty} \frac{\log \pAoN}{\log \pmmt}=1,\]where $\pmmt$ is the first moment threshold defined in \eqref{def:p1M}. In other words, $p_{\mathrm{AoN}}=(\pmmt)^{1+o_N(1)}.$
\end{theorem}

In preparation for the proof, let $m_i$ be the posterior probability of $i\in\bS$, given $\bY=\bS\cup\bV$:
	\beq\label{e:calculate.m.i}
	m_i\equiv m_i(\bY)=\E_{\P}(\theta_i\,|\,\bY)=\P(i\in\bS\,|\,\bY)
	= \frac{\bZ_{\ni i}(\bY)}{\bZ(\bY)}\,,
	\eeq
where $\bZ_{\ni i}(\bY)$ denotes the number of subsets ${\BG S' \in \bS}$ \BG that are contained $\bY$ with $i\in{\BG S'}$. \EH Note that if $i\notin\bY$ then 
$\bZ_{\ni i}(\bY)=0$, and therefore $m_i=0$. Note also that
	\beq\label{e:m.i.sum.K}
	\sum_{i=1}^N m_i = K
	\eeq
with probability one, since the prior is uniform on sets of size $K$.
The next lemma gives a bound on the derivative of $D(p)$ in terms of the MMSE of the model:

\begin{lemma}\label{l:kl.divergence.derivative}
Let $D(p)\equiv\DKL(\P_p\,|\,\Q_p)$ be the Kullback--Leibler divergence between the laws of $\bY$ under the planted and null models. Then
	\[
	\frac{d}{dx}D(e^{-x})
	\le K-\mmse\,.
	\]
\end{lemma}

\begin{proof}
Abbreviate $\P=\P_p$ and $\Q=\Q_p$.
Recalling \eqref{e:radon.nikodym}, we have
	\beq\label{e:D.p}
	D(p) 
	= \DKL(\P\,|\,\Q)
	= \sum_{\bY} \P(\bY) \log \frac{\P(\bY)}{\Q(\bY)}
	= \E_{\P} \log \frac{\bZ(\bY)}{Mp^K}
	= \E_{\P}  \log \frac{\bZ(\bS\cup\bV)}{M p^K}\,,
	\eeq
where $\bS\sim\bP$ and $\bV\sim\Bern(p)^{\otimes N}$. Next note that for any function $f(\bV)$, we have
	\begin{align}\nonumber
	\frac{d}{dp} \E_{\P} f(\bV)
	&= \E_{\P}\bigg[ f(\bV) \sum_{i=1}^N
		\bigg( \frac{w_i}{p}-\frac{1-w_i}{1-p}\bigg) \bigg]
	= \sum_{i=1}^N
	\E_{\P} \bigg(  f(\bV\cup\{i\})
	- f(\bV \setminus\{i\})\bigg)\\
	&= \frac{1}{p}
		\E_{\P}\bigg[ f(\bV) \sum_{i=1}^N
		\bigg(1-\frac{1-w_i}{1-p}\bigg)\bigg]
	= \frac{1}{p}
	\sum_{i=1}^N 
	\E_{\P} \bigg( f(\bV) - f(\bV\setminus\{i\}) \bigg)\,.
	\label{e:derivative.Ef}
	\end{align}
In the above, $w_i$ refers to the indicator that $i\in\bV$. Now consider $f(\bV)=\log \bZ(\bS\cup\bV)$. If $i\in\bS$, then $f(\bV)$ does not depend on $w_i$. If $i\notin\bS$, then $f(\bV)$ depends on $w_i$, and we have
	\[
	f(\bV)-f(\bV\setminus\{i\})
	= \log \frac{\bZ_{\ni i}(\bS\cup\bV)
	+ \bZ(\bS\cup\bV\setminus\{i\})}{\log \bZ(\bS\cup\bV\setminus\{i\})}
	= \log \frac{\bZ(\bY)}{\log \bZ(\bY\setminus\{i\})}
	\stackrel{\eqref{e:calculate.m.i}}{=}
	\frac{1}{1-m_i(\bY)}\,.
	\]
%	\[
%	f(\bV\cup\{i\})-f(\bV\setminus\{i\})
%	= \log \frac{\bZ_{\ni i}(\bS\cup\bV)
%	+ \bZ(\bS\cup\bV\setminus\{i\})}{\log \bZ(\bS\cup\bV\setminus\{i\})}
%	= \log \frac{\bZ(\bS\cup\bV)}{\log \bZ(\bS\cup\bV\setminus\{i\})}
%	\,.
%	\]
Substituting this into \eqref{e:derivative.Ef} gives
	\[
	\frac{d}{dp} \E_{\P}\log\bZ(\bS\cup\bV)
	=\E_{\P}\sum_{i\notin\bS}
	\log \frac{1}{1-m_i(\bY)}\,.
	\]
Then, using that $m+(1-m)\log(1-m) \le m^2$ for all $m\in[-1,1]$, we have
	\begin{align}\nonumber
	\frac{d}{dp} D(p)
	&= \frac1p \E_{\P} \sum_{i\notin \bS} \log \frac1{1-m_i(\bY)}
	- \frac{K}{p}
	\stackrel{\eqref{e:m.i.sum.K}}{=}
	\frac1p \E_{\P} \sum_{i=1}^N \bigg\{
	 (1-m_i)
	\log \frac{1}{1-m_i} - m_i \bigg\} \\
	&\ge
	- \frac1p\E_{\P} \sum_{i=1}^N (m_i)^2
	\stackrel{\eqref{e:m.i.sum.K}}{=} 
	\frac1p\bigg( \E_{\P} \sum_{i=1}^N m_i(1-m_i)
	-K\bigg)
	= \frac1p (\textup{MMSE}_N(p)-K)\,.
	\label{I-MMSE_1sd}
	\end{align}
The claim follows by making the change of variables $p=e^{-x}$.
\end{proof}

\begin{corollary}\label{c:paon.above.pc}
If the planted model $\P_p$ has all-or-nothing at the exponential scale at threshold $\pAoN$, then we must have $\pAoN \ge (\pmmt)^{1+o_N(1)}$ for $\pmmt$ as in \eqref{def:p1M}.
\end{corollary}

\begin{proof}
Take $\delta>0$ and let $p_0=(\pAoN)^{1-\delta}\equiv \exp(-x_0)$. Since the MMSE is nondecreasing in $p$, we have $\mmse/K\to1$ uniformly over all $[p_0,1]$. Note also that $D(1)=\DKL(\P_1\,|\,\Q_1)=1$, so Lemma~\ref{l:kl.divergence.derivative} gives
	\[
	D(p)
	= D(e^{-x})
	= \int_0^{x} \frac{d}{dt} D(e^{-t}) \,dt
	\le o_N(1) Kx
	= o_N(1) K \log \frac{1}{p}\,,
	\]
uniformly over $p\in[p_0,1]$.
On the other hand, since $\bZ(\bY)\ge1$ always, it follows from \eqref{e:D.p} that
	\[
	D(p) \ge \log\frac{1}{M p^K}
	\stackrel{\eqref{def:p1M}}{=}
	K \bigg\{
	\log \frac{1}{p}-\log \frac{1}{\pmmt} 
	\bigg\}
	\]
for all $p\in[0,1]$.
Combining the above bounds gives
	\[\log \frac{1}{\pmmt}
	\ge (1-o_N(1))\log \frac{1}{p_0}
	= (1-o_N(1)) (1-\delta)
	\log \frac{1}{\pAoN}
	\,,
	\]
and the claim follows by sending $\delta\downarrow0$.
\end{proof}

The next lemma says that under the \BM RWS \EH
condition,
in the ``all'' regime the mutual information between $\P_p$ and $\Q_p$ must be close to its maximal value $\log M$:

\begin{lemma}\label{lem:ent_post}
Let $I(p)\equiv I(\bS;\bY)=H(\bS)-H(\bS\,|\,\bY)$ be the mutual information between $\bS$ and $\bY=\bS\cup\bV$ under the planted model $\P_p$. 
Suppose the prior $\bP$ satisfies the RWS condition (Definition~\ref{d:rws}).
If the model is in the ``all'' regime in the sense that $\mmse/K = o_N(1)$, then
	\[
	\lim_{N\to\infty} \frac{I(p)}{\log M}
	= 1-\lim_{N\to\infty} \frac{\log H(\bS\,|\,\bY)}{\log M}
	= 1\,.
	\]
\end{lemma}

\begin{proof}It follows using \eqref{e:P.vs.Q} and \eqref{e:radon.nikodym} that with $\P=\P_p$, we have
	\beq\label{e:I.p}
	I(p)
	=\sum_{\bS,\bY}
	\P(\bS,\bY)
	\log\frac{\P(\bS,\bY)}{\bP(S) \P(\bY)}
	= \E_{\P} \log\frac{M}{\bZ(\bY)}\,,
	\eeq
where we note that under the planted model we have $1\le \bZ(\bY)\le M$ with probability one. The assertion of the lemma can then be rewritten as
	\[
	1-\frac{I(p)}{\log M}
	= \frac{\E_{\P}\log\bZ(\bY)}{\log M}
	= o_N(1)\,.
	\]
We then decompose $\bZ(\bY)= \bZ_\circ(\bS,\bY)+\bZ_\bullet(\bS,\bY)$ where
	\[
	\bZ_\circ(\bS,\bY)
	= \sum_{\bS'\in\cS}
	\mathbf{1}\Big\{\bS'\subseteq\bY, |\bS\cap\bS'| < K(1-\delta)
		\Big\}\,,
	\]
and $\bZ_\bullet(\bS,\bY)$ is the remainder. The ``all'' assumption implies 
that for any fixed $\delta>0$ we have
	\[
	\frac{\bZ_\circ(\bS,\bY)}{\bZ(\bS,\bY)} \to0
	\]
in probability as $N\to\infty$. We then crudely bound
	\[
	1\le \bZ_\bullet(\bS,\bY)
	\le
	\sum_{\bS'\in\cS}
	\mathbf{1}\Big\{|\bS\cap\bS'| \ge K(1-\delta)
		\Big\}
	= M \bP
	\bigg(|\bS \cap \bS'| \ge (1-\delta)K \,\bigg| \,\bS\bigg)\,.
	\]
Combining with the \BM RWS \EH assumption~\eqref{e:rws} gives
	\[
	\adjustlimits
	\lim_{\delta\downarrow0}\limsup_{N\to\infty}
	\frac{\E_{\P}\log\bZ_\bullet(\bS,\bY)}{\log M}
	\le \adjustlimits
	\lim_{\delta\downarrow0}\limsup_{N\to\infty}
	\frac{\log \E_{\P}\bZ_\bullet(\bS,\bY)}{\log M}
	\stackrel{\eqref{e:rws} }{=}0\,.
	\]
It follows by combining the above bounds that
$\bZ_\circ(\bS,\bY)
\ll \bZ_\circ(\bS,\bY)\le M^{o_N(1)}$,
and consequently we have $\bZ(\bY)\le M^{o_N(1)}$,
with high probability under $\P$.
\end{proof}

\begin{corollary}\label{c:paon.below.pc}
Suppose the prior $\bP$ satisfies the RWS condition (Definition~\ref{d:rws}). If the planted model $\P_p$ has all-or-nothing at the exponential scale at threshold $\pAoN$, then we must have $\pAoN \le (\pmmt)^{1-o_N(1)}$ for $\pmmt$ as in \eqref{def:p1M}.
\end{corollary}

\begin{proof}
Since $\bZ(\bY)\ge1$ always, it follows from \eqref{e:I.p} that for all $p\in[0,1]$ we have
	\[
	I(p) \ge \log \frac{1}{p^K}\,.
	\]
Take $\delta>0$ and let $p_0=(\pAoN)^{1+\delta}$: this is in the ``all'' regime, so combining with Lemma~\ref{lem:ent_post} gives
	\[
	(1-o_N(1))\log M = I(p_0) \ge \log \frac{1}{(p_0)^K} 
	= K(1+\delta)\log\frac{1}{\pAoN}\,.
	\]
Dividing through by $K$ gives
	\[
	(1-o_N(1))\log \frac{1}{\pmmt}
	\stackrel{\eqref{def:p1M}}{=}
	(1-o_N(1))\log M^{1/K}
	\ge(1+\delta)\log\frac{1}{\pAoN}\,,
	\]
and the claim follows by sending $\delta\downarrow0$.
\end{proof}

\begin{proof}[Proof of Theorem \ref{thm:locating}]
Follows by combining Corollaries \ref{c:paon.above.pc} and \ref{c:paon.below.pc}.
\end{proof}

\subsection{First-moment-stable graphs satisfy RWS}
\label{ss:first.mmt.stable}

Recall from Definition~\ref{d:first.mmt.stable} that a graph sequence $H=(H_n)_{n\ge1}$ is \emph{first-moment-stable} if 
for all $J\subseteq H$ with $|J|\ge |H|(1-o(1))$ we have
$\pmmt(J) = \pmmt(H)^{1+o(1)}$. Recall also the notations $M_{J,H}$ and $M_{H|J}$ from \eqref{e:double.count}. In this subsection we give an equivalent characterization of first-moment-stability, 
and also show that it implies RWS.

%\begin{definition}\label{d:first.mmt.stable}
%We say the graph sequence $H=(H_n)_{n\ge1}$ is \emph{\BM first-moment-stable} if 
%	\[
%	\adjustlimits 
%	\lim_{\delta\downarrow0}
%	\limsup_{N\to\infty}
%	\sup\bigg\{
%	\bigg|\frac{\log \pmmt(J)}{\log \pmmt(H)}-1\bigg|
%	: J \subseteq H, |J| \ge |H|(1-\delta)
%	\bigg\}=0\,.
%	\]
%Informally this says that if $|J|=|H|(1-o(1))$ then we must have
%$\pmmt(J) = \pmmt(H)^{1+o(1)}$. We will see below that
%it is equivalent to require
%$\BM\pmmt(J) \le \pmmt(H)^{1-o(1)}$;
%the other direction holds automatically.
%\end{definition}

\begin{lemma}[alternate characterization of first-moment-stability]
\label{l:M.given.H.right.cts}
If $\pmmt(H)$ is bounded away from one, then 
$(H_n)_{n\ge1}$ is first-moment-stable in the sense of Definition~\ref{d:first.mmt.stable}
if and only if $M_{H|J} \le (M_H)^{o_\delta(1)}$
uniformly over $J\subseteq H$ with $|J|\ge |H|(1-\delta)$.
\end{lemma}

\begin{proof}
Recall from \eqref{def:p1M} that the first moment threshold of $J$ is given by 
	\[
	\pmmt(J) = \frac{1}{(M_J)^{1/|J|}}\,.
	\]
In particular, note that $\pmmt(H)$ is bounded away from one if and only if $|H|=O(\log M)$. Next, for all subgraphs $J\subseteq H$ with $|J|\ge |H|(1-\delta)$, it holds uniformly that
	\[
	1\le M_{J,H}
	\le 
	\binom{|H|}{|J|}
	\le \exp(|H| o_\delta(1))
	\le (M_H)^{o_\delta(1)}\,,
	\]
having used the assumption $|H|=O(\log M_H)$.
Now recall from \eqref{e:double.count} that
	\[
	\frac{\pmmt(J)^{|J|}}{\pmmt(H)^{|H|}}
	= \frac{M_H}{M_J}
	\stackrel{\eqref{e:double.count}}{=} \frac{M_{H|J}}{ M_{J,H} }\
	= \frac{M_{H|J}}{ (M_H)^{o_\delta(1)}}
	\ge \frac{1}{ (M_H)^{o_\delta(1)}}\,.
	\]
It follows that for $J\subseteq H$ with $|J|\ge |H|(1-o(1))$,
we have
 $\pmmt(J)\ge \pmmt(H)^{1+o(1)}$ always,
and we have $\pmmt(J)\le \pmmt(H)^{1-o(1)}$ (hence first-moment-stability) if and only if $M_{H|J} \le (M_H)^{o_\delta(1)}$. 
\end{proof}

\begin{corollary}[first-moment-stability implies RWS]
\label{c:right.cts.implies.weak.sep}
Let $\bP$ be the uniform prior on all copies of $H_n$ in $K_n$.
If $\pmmt(H)$ is bounded away from one, and $(H_n)_{n\ge1}$ is first-moment-stable in the sense of Definition~\ref{d:first.mmt.stable}, then $\bP$ satisfies RWS in the sense of Definition~\ref{d:rws}.
\end{corollary}

\begin{proof} 
We can bound
	\[\bP^{\otimes 2}\Big(|\bH \cap \bH'| \ge (1-\delta)K \Big) 
	M
	\le
	\binom{|H|}{|H|(1-\delta)}
	\sup\bigg\{
	 M_{H|J}
	 : J\subseteq H, |J| \ge |H|(1-\delta)\bigg\}\,,
	\]
so the claim follows by applying Lemma~\ref{l:M.given.H.right.cts} and sending $\delta\downarrow0$.
\end{proof}

Thus, Corollary~\ref{c:right.cts.implies.weak.sep} tells us that Theorem~\ref{thm:locating} applies to first-moment-stable graphs.

%
%
%
%We also impose a mild technical assumption, that
%	\beq\label{e:good}
%	\min\bigg\{|H|,v(H) \log v(H)\bigg\}
%	=o(\log M_H)
%	\eeq
%The next lemma gives conditions for \eqref{e:good} to hold:
%\begin{definition}\label{d:good}
%We call a graph $H$ \emph{\BM good} if  \iz{We need a better name...}
%	\[\min\bigg\{|H|,v(H) \log v(H)\bigg\}
%	=o(\log M_H)\,.\]
%\end{definition}

%\begin{lemma}\label{lem:good}
%Suppose at least one of the following conditions is satisfied.
%\begin{itemize}
%\item[(a)] $p_G=o(1)$.
%\item[(b)]  $v(G) \log v(G)=o(e(G))$ and $\limsup p_G<1$.
%%\item[(c)] For some $\delta>0$, $|\mathrm{Aut}(G)| \leq v(G)!^{1-\delta}$
%\end{itemize}Then $\min\{v(G)\log v(G), e(G)\}=o(\log M_G)$, i.e. $G$ is good.
%\end{lemma}
%
%

%Finally, recall that when planting a copy of $G$ in $G(n,p)$ we say that 
%\textbf{AoN happens at the exponential scale} if for some critical $\pmmt \in (0,1)$
% \begin{equation}\label{eq:exp_all_or_nothing}
%\lim_{n \to \infty}  \frac{1}{e(G)} \mathrm{MMSE}_n(G,p)  = \left\{
%\begin{array}{ll}
%1 & \text{ if $p>(\pmmt)^{1-\epsilon}$} \\
%0 & \text{ if $p<(\pmmt)^{1+\epsilon}$}\,.
%\end{array}
%\right.
%\end{equation}

\subsection{Theorem~\ref{t:aon.exp} forward direction: AoN implies first-moment-flat}
\label{ss:exp.forward}

We now turn to the proof of Theorem~\ref{t:aon.exp}.
The result holds for the following class of graph sequences: all graphs that have first moment threshold tending to zero, or that contain a subpolynomial number of vertices. To this end we introduce the following useful technical condition:
	\beq\label{e:good}
	\min\bigg\{|H|,v(H) \log v(H)\bigg\}
	=o(\log M_H)
	\eeq
The next lemma explains that \eqref{e:good} holds exactly in the cases mentioned above:

\begin{lemma}\label{l:good}
We have the following equivalences:
\begin{enumerate}[(a)]
\item $\pmmt(H)=o(1)$ if and only if $|H| \ll \log M_H$;
\item $v(H) \le n^{o(1)}$ if and only if $v(H)\log v(H) \ll \log M_H$.
\end{enumerate}
If either of the above holds, then $H$ satisfies \eqref{e:good}.
\end{lemma}

\begin{proof}
The first claim (a) follows immediately from \eqref{e:graph.first.mmt.threshold}.
For the second claim (b), 
note that the upper bound in
\eqref{e:trivial.bounds.on.M} says $\log M_H \le v(H)\log n$, so 
$v(H)\log v(H) = o(\log M_H)$ implies $v(H) \le n^{o(1)}$. 
Conversely, if $v(H) \le n^{o(1)}$, then combining with the lower bound in 
\eqref{e:trivial.bounds.on.M} gives
	\[
	\log M_H
	\gtrsim v(H) \log \frac{n}{v(H)}
	\gg v(H) \log v(H)\,.
	\]
This proves (b).
\end{proof}

Recall that we assume $\pmmt(H)$ is bounded away from one, or equivalently $|H|=O(\log M_H)$. As a consequence, Lemma~\ref{l:good} case (b) includes dense graphs (Definition~\ref{d:dense}), since 
for such graphs we have
	\[
	\frac{v(H) \log v(H)}{\log M_H}
	\ll \frac{|H|}{\log M_H} \le O(1)\,.
	\]
However, Theorem~\ref{t:aon.exp} goes beyond dense graphs, \BM as discussed above. \EH

\begin{proof}[Proof of Theorem~\ref{t:aon.exp} forward direction]
We argue by contradiction. Assume AoN occurs at the exponential scale, but the graph is not first-moment-flat. This means that for some $0<\delta<q<1$ there must exist a subgraph $J\subseteq H$ with $|J|\ge |H|\delta$ and
	\[
	%\pmmt(H)^{1-\delta} 
	\frac{1}{(\exp(\lambda_q))^{1-\delta}}
	\le \pmmt(J)^{1+\delta} \,.
	\]
Let $0<\epsilon \ll \delta$. Since
$H$ is first-moment-stable, Corollary~\ref{c:right.cts.implies.weak.sep} gives that the uniform measure $\bP$ on copies of $H$ in $K_n$ satisfies RWS (Definition~\ref{d:rws}). It then follows by Theorem~\ref{thm:locating}
(in fact by Corollary~\ref{c:paon.below.pc}) that
$p_{\mathrm{AoN}} 
\le (\pmmt(H))^{1-o(1)}$.
Moreover, the first-moment-stable assumption gives 
	\beq\label{e:appl.right.cts}
	\lim_{q\uparrow1} \lambda_q
	= \lambda_1= \log \frac{1}{\pmmt(H)}
	=\frac{ \log M_H}{|H|}
	\,,
	\eeq
so for $q$ sufficiently close to one we will also have
	\[p_{\mathrm{AoN}} 
	\le \frac{1}{(\exp(\lambda_q))^{1-\epsilon}}\,.\]
%We consider two cases:
%\begin{enumerate}[(i)]
%\item If $H$ is first-moment-stable, then Corollary~\ref{c:right.cts.implies.weak.sep} gives that the uniform measure $\bP$ on copies of $H$ in $K_n$ satisfies RWS (Definition~\ref{d:rws}). It then follows by Theorem~\ref{thm:locating}
%(in fact by Corollary~\ref{c:paon.below.pc}) that
%$p_{\mathrm{AoN}} 
%\le (\pmmt(H))^{1-o(1)}$.
%Moreover, the first-moment-stable assumption gives 
%	\[
%	\lim_{q\uparrow1} \lambda_q
%	= \lambda_1= \log \frac{1}{\pmmt(H)}\,,
%	\]
%so for $q$ sufficiently close to one we will also have
%	\[p_{\mathrm{AoN}} 
%	\le \frac{1}{(\exp(\lambda_q))^{1-\epsilon}}\,.\]
%
%\item If $v(H)\log v(H)\ll\log M_H$, then Lemma~\ref{l:not.all.exp}
%tells us that ``all'' cannot occur 
%for $p \ge (\exp(-\lambda_q))^{1-\epsilon}$ for $q\in(0,1)$. Since we assume AoN, it means that ``nothing'' occurs in this regime, so
%	\[p_{\mathrm{AoN}} 
%	\le \frac{1}{(\exp(\lambda_q))^{1-\epsilon}}\]
%for $q\in(0,1)$.
%\end{enumerate}
Consequently, for $q$ close to one,\footnote{In fact, in light of \eqref{e:appl.right.cts} we can take $q=1$. However in this argument we use $q<1$ to highlight the places where the first-moment-stability assumption is required.} we can take $p$ to satisfy
	\beq\label{e:choice.of.p}
	\pAoN 
	\le \frac{1}{(\exp(\lambda_q))^{1-\epsilon}}
	\le p  
	\le \frac{1}{(\exp(\lambda_q))^{1-\delta}}
	\le \pmmt(J)^{1+\delta}\,.\eeq
Now recall that $\bH$ denotes the planted copy of $H$.
Since $p$ is in the ``nothing'' regime, with high probability the observed graph $\bY=\bH\cup\bG$ contains another copy $H'$ of $H$ which has negligible overlap with planted copy $\bH$, in the sense that
	\[
	\frac{|\bH\cap H'|}{|\bH'|} = o(1)\,.
	\]
It follows that under the null model $\bG\sim G(n,p)$ contains, with high probability, an \emph{approximate copy} of $H$ --- that is, a subgraph on $v(H)$ vertices which can be made into a copy of $H$ by adding at most $o(|H|)$ edges. %Since $|J|\ge |H|\delta$, 
As a consequence, $\bG$ also contains with high probability an approximate copy of the graph $J$ from \eqref{e:choice.of.p}, since
$|J|\ge |H|\delta$.

Let $\bZ_\eta=\bZ_{J,\eta}(\bY)$ count the total number of $\eta$-approximate copies of $J'\approx J$ contained in $\bY$ --- that is, subgraphs on $v(J)$ vertices which can be made into a copy of $J$ by adding at most $\eta|J|$ edges. The preceding argument shows that
$\bZ_\eta\ge1$ with high probability under the null model $\Q=\Q_p$.
We will derive a contradiction by showing that $\E_{\Q} \bZ_\eta<1$. To this end note
	\[
	\E_{\Q} \bZ_\eta
	= \sum_{J'\approx J}
	M_{J'} p^{|J'|}
	\le
	\sum_{J'\approx J}
	M_{J'} p^{|J|(1-\eta)}
	\,.
	\]
The number of $J'\approx J$ is $\exp(|J|o_\eta(1))$. On the other hand, for any $J'\subseteq J$ with $|J'|\ge |J|(1-\eta)$, 
	\[
	M_{J'}
	= \frac{M_J M_{J',J}}{M_{J|J'}}
	\le M_J M_{J',J}
	\le M_J \exp( |J| o_\eta(1))\,.
	\]
(this is the same reasoning as the easy direction of Lemma~\ref{l:M.given.H.right.cts}.) It follows from \eqref{e:choice.of.p} that
	\[
	M_J
	=\frac{1}{(\pmmt(J))^{|J|}}
	\le\bigg( \frac{1}{p^{|J|}}\bigg)^{1/(1+\delta)}
	\le \frac{1}{p^{|J|(1-2\delta/3)}}\,.
	\]
Combining these bounds gives
	\[
	\E_{\Q} \bZ_\eta
	\le
	\exp(|J|o_\eta(1))
	%\bigg( \frac{1}{p^{|J|}}\bigg)^{1+2\delta/3}
	\frac{p^{|J|(1-\eta)} }{p^{|J|(1-2\delta/3)}}
	\le
	\exp(|H|o_\eta(1))
	p^{|J|\delta/2}
	\,,
	\]
where the last inequality holds by taking $\eta\ll\delta$. Recalling \eqref{e:choice.of.p} again now gives
	\[
	\E_{\Q} \bZ_\eta
	\le
	\frac{\exp(o_\eta(1) \log M_H)}
	{(\exp(\lambda_q))^{|J|\delta(1-\delta)/2}}\,.
	\]
To conclude, recall from \eqref{e:appl.right.cts} that
for $q$ sufficiently close to one we will have
$|H|\lambda_q \gtrsim \log M_H$, which allows us to conclude that the above bound is $o_N(1)$. \EH
\end{proof}

\subsection{Theorem~\ref{t:aon.exp} reverse direction: first-moment-flat implies AoN}
\label{ss:exp.reverse}

\begin{proof}[Proof of Theorem~\ref{t:aon.exp} reverse direction]
Since we assumed $|H_n|\to\infty$ and $\pmmt(H_n)$ bounded away from one, it follows that
$\log |H| \ll |H| \le O(\log M_H)$, so we can apply
Theorem~\ref{t:bern.aon.exp}.
Hence it suffices to check the growth condition \eqref{e:weak.growth}, that is,
\beq\label{weak_growth_2}
\limsup_{n\to\infty}
\sup_{0\le\ell\le |H|} \frac{1}{\log M_H}
\log
\bigg(  
\frac{\bP^{\otimes 2}(|\bH\cap \bH'| = \ell)}{\pmmt(H)^\ell}
\bigg)
 \leq 0\,,
\eeq where $G_1$ and $G_2$ are i.i.d.\ draws from $\bP$, the uniform measure over all copies of $G$ in $K_n$.

Now notice that for $\ell \ll |H|$ the bound \eqref{weak_growth_2} is immediate, since in this case
	\[\frac{\bP^{\otimes 2}(|\bH\cap \bH'| = \ell)}{\pmmt(H)^\ell}
	\le \frac{1}{\pmmt(H)^\ell}
	= (M_H)^{\ell/|H|} \le  (M_H)^{o(1)}\,.
	\]
It therefore suffices to bound the case $\delta|H| \le\ell\le |H|$. Let $H_0$ be any fixed copy of $H$ in $K_n$, and note that
	\[
	\bP^{\otimes 2}\Big(|\bH\cap \bH'| = \ell\Big)
	\le
	\sum_{J_0\in I_\ell(H_0)}
	\bP( H_0\cap \bH = J_0)
	\]
where $I_\ell(H_0)$ is the set of all subgraphs $J_0\subseteq H_0$ with $|J_0|=\ell$ that can arise as an intersection of $H_0$ with another copy of $H$. We then bound
	\[
	\bP( H_0\cap \bH = J_0)
	\le \bP(J_0\subseteq\bH)
	= \frac{M_{H|J}}{M_H} = \frac{M_{J,H}}{M_J}\,.
	\]
By the first-moment-flat condition (Definition~\ref{d:first.mmt.flat}) together with the first-moment-stable assumption,
	\[
	\frac{\log M_J}{|H|} \ge (1-o(1)) \frac{\log M_H}{H}\,,
	\]
so we can bound
	\[
	\frac{1}{M_J}
	\le \bigg( \frac{1}{M_H}\bigg)^{(1-o(1)) \ell/|H|}
	\le (M_H)^{o(1)} \pmmt(H)^\ell \,.
	\]
Combining these bounds and rearranging gives
	\[\frac{\bP^{\otimes 2}(|\bH\cap \bH'| = \ell)}{\pmmt(H)^\ell }
	\le (M_H)^{o(1)} 
	\sum_{J_0\in I_\ell(H_0)} M_{J,H}
	\le (M_H)^{o(1)} 
	|I_\ell(H)|^2
	\]
To finish the proof, we now claim that
	\[\log |I_\ell(H)|
	\le O\bigg(
	\min\Big\{ e(H), v(H)\log v(H)\Big\}\,.
	\bigg)\,.\]
That the left hand side is at most $O(e(H))$ is clear. For the other part of the inequality, recall that a graph in $I_\ell(H)$ must be realized
 as a \BH (vertex-induced) \EH intersection of two copies of $H$. Hence, choosing the isomorphism class of the subgraph of $H$ ($\exp(O(v(H)))$ choices), the vertices used ($\exp(O(v(H)))$ choices), and the way to embed the graph in these vertices ($\exp(O(v(H)\log v(H)))$ choices) implies the desired result. By the assumption \eqref{e:good}, 
we have
	\[\frac{\log I_\ell(H)}{\log M_H}
	=o(1)\]
for all $\ell \geq \delta |H|$. This completes the proof.
\end{proof}

\section{Conclusion}
In this work we considered the model of a general subgraph $H=H_n$ planted in $G(n,p)$. We showed that, under various assumptions on $H$, the AoN phenomenon in the planted model
can be characterized in terms of the ``generalized expectation thresholds'' of $H$ in the null model $G(n,p)$. (See Theorems \ref{t:dense.linear.aon} and \ref{t:aon.exp} for the precise statements.) A natural question would be whether an AoN characterization can be obtained \emph{for all planted subgraphs $H$}. In a more general context, our results, alongside with the intuition described in Section~\ref{ss:intuition}, suggest that AoN can be characterized by merely studying structural properties of the ``solution space'' in the null model (corresponding, e.g., to the absence of a ``condensation phase'' in the language of random constraint satisfaction problems \cite{krzakala2007gibbs}). It would be interesting to investigate further this connection.

Lastly, as indicated above, sharp thresholds in boolean Fourier analysis have been long conjectured to be connected with computational hardness, see e.g.\ \cite{kalai2006perspectives}. A prime example of such a connection is the fact that bboolean circuits of ``low complexity'' do not exhibit sharp threshold behavior \cite[\S6]{kalai2006perspectives}. Meanwhile, on the inference side,
a large amount of work in the past decade has been devoted to studying the existence of ``computational-statistical'' gaps: regimes where the inference task is information-theoretically possible, but appears intractable by efficient algorithms.
Intriguingly, AoN (the inference analogue of sharp thresholds) has been empirically observed to appear (with a few puzzling exceptions) in models with a computational-statistical gap.  For instance, we have seen that AoN appears for the planted clique model (Corollary~\ref{c:planted.clique}), but not for the planted matching problem (Example~\ref{x:matching}).
Correspondingly, there is a substantial body of evidence towards a 
computational-statistical gap in the planted clique problem (e.g., \cite{barak2019nearly, feldman2017statistical,gamarnik2019landscape}), but the planted matching problem does not exhibit such a gap (the maximum matching is polynomial-time computable, and gives non-trivial recovery up to the information-theoretic threshold \cite{MR4350971,ding2021planted}). This leads us to ask:
\begin{quotation}
   \centering{ \emph{Is AoN a provable barrier for a subclass of polynomial-time methods?}}
\end{quotation} We consider this a natural and intriguing question for future work.

\section*{Acknowledgements}
We acknowledge the
support of Simons-NSF grant DMS-2031883 (E.M., Y.S., N.S., and I.Z.), the Vannevar Bush Faculty Fellowship
ONR-N00014-20-1-2826 (E.M., Y.S., and I.Z.), the Simons Investigator Award 622132 (E.M.), the Sloan Research
Fellowship (J.N.W.), NSF CAREER grant DMS-1940092 (N.S.), and the Solomon Buchsbaum Research Fund at MIT (N.S.).

\bibliographystyle{alphaabbr}
{\raggedright\bibliography{pc}}
\end{document}